\documentclass[a4paper]{amsart}
\usepackage{graphicx}
\usepackage{amsfonts}
\usepackage{amssymb}
\usepackage{amscd}
\usepackage{enumerate}
\usepackage{mathrsfs}%French style
\usepackage{braket}%bracket

\usepackage[hang, small,bf]{caption}
\usepackage[subrefformat=parens]{subcaption}%labelformat=simple
\makeatletter
\renewcommand{\section}{%
\@startsection{section}{1}%
  \z@{.7\linespacing\@plus\linespacing}{.5\linespacing}%
 {\normalfont\large\bfseries\centering}}
\makeatother

\newtheorem{theorem}{Theorem}[section]
\newtheorem{corollary}[theorem]{Corollary}

\newtheorem{lemma}[theorem]{Lemma}
\newtheorem{proposition}[theorem]{Proposition}

\newtheorem*{theorem*}{Theorem}
\newtheorem*{corollary*}{Corollary}
\newtheorem*{conjecture*}{Conjecture}
\newtheorem*{lemma*}{Lemma}
\newtheorem*{proposition*}{Proposition}
\newtheorem*{problem*}{Problem}
\newtheorem*{axiom*}{Axiom}
\newtheorem*{example*}{Example}
\newtheorem*{exercise*}{Exercise}

\theoremstyle{definition}%remark definition style
\newtheorem{remark}[theorem]{Remark}
\newtheorem*{remark*}{Remark}
\newtheorem{definition}[theorem]{Definition}
\newtheorem*{definition*}{Definition}
\renewcommand{\l}{\left}
\renewcommand{\r}{\right}
\newcommand{\eps}{\varepsilon}
\newcommand{\N}{{\mathbb N}}
\newcommand{\R}{{\mathbb R}}
\newcommand{\C}{{\mathbb C}}

\newcommand{\im}{\operatorname{Im}}
\newcommand{\re}{\operatorname{Re}}

\newcommand{\ds}{\displaystyle}

\newcommand{\x}{\times}
\newcommand{\del}{\partial}
\newcommand{\cleq}{\lesssim}

\newcommand{\wto}{\rightharpoonup}%weak conv.
\newcommand{\til}{\widetilde}

\def\norm[#1]{\left\Vert #1 \right\Vert}
\def\tbra[#1,#2]{\left\langle #1 , #2\right\rangle} %inner product%trianguler bracket <\UTF{0081}E>
\def\rbra[#1,#2]{\left( #1 , #2 \right)} %inner product% round bracket\UTF{0081}@(\UTF{0081}E)
\def\sbra[#1,#2]{\left[ #1 , #2 \right]} %inner product% square bracket [\UTF{0081}E]

\newcommand{\scA}{{\mathscr A}}
\newcommand{\scB}{{\mathscr B}}

\newcommand{\scG}{{\mathscr G}}

\newcommand{\scK}{{\mathscr K}}

\newcommand{\scM}{{\mathscr M}}
\newcommand{\scN}{{\mathscr N}}

\begin{document}

\title[]{Instability of algebraic standing waves for nonlinear Schr\"odinger equations with double power nonlinearities}
%On two types of standing waves for nonlinear Schr\"odinger equations with double power nonlinearities
%Strong instability for algebraic solitary waves of nonlinear Schr\"odinger equations with double power nonlinearities


%    author one information

\author[N. Fukaya]{Noriyoshi Fukaya}
\address{Department of Mathematics, Tokyo University of Science, Tokyo, 162-8601, Japan}
\email{fukaya@rs.tus.ac.jp}

\author[M. Hayashi]{Masayuki Hayashi}
\address{Research Institute for Mathematical Sciences, Kyoto University, Kyoto 606-8502, Japan}
%Department of Applied Physics, Waseda University, Tokyo 169-8555, Japan
\curraddr{}
\email{hayashi@kurims.kyoto-u.ac.jp}
%masayuki-884@fuji.waseda.jp
\thanks{}

%    \subjclass is required.
%35A15  	Variational methods
%35B35  	Stability
%35Q51  	Soliton-like equations
%35Q55  	NLS-like equations
%35C07  	Traveling wave solutions

\date{}

\dedicatory{}

%    "Communicated by" -- provide editor's name; required.
% \commby{}

\begin{abstract}
We consider a nonlinear Schr\"odinger equation with double power nonlinearity
\begin{align*}
i\del_t u+\Delta u-|u|^{p-1}u+|u|^{q-1}u=0,\quad (t,x)\in\R\times\R^N,
\end{align*}
where $1<p<q<1+4/(N-2)_+$. Due to the defocusing effect from the lower power order nonlinearity, the equation has algebraically decaying standing waves with zero frequency, which we call \textit{algebraic standing waves}, as well as usual standing waves decaying exponentially with positive frequency. In this paper we study stability properties of two types of standing waves. We prove strong instability for all frequencies when $q\ge 1+4/N$ and instability for small frequencies when $q<1+4/N$, which especially give the first results on stability properties of algebraic standing waves. The instability result for small positive frequency when $q<1+4/N$ not only improves previous results in one-dimensional case but also gives a first result on instability in higher-dimensional case. The key point in our approach is to take advantage of algebraic standing waves.
\end{abstract}
%Our results give a first result on stability properties of algebraic standing waves. The instability result for the case $q<1+4/N$ not only improves previous results in one-dimensional case and but also gives a first result on instability in higher-dimensional case.
%Our instability result for the case $q<1+4/N$ improves the previous results in one-dimensional case and gives a first result on instability in higher-dimensional case. 
%For the proof we take advantage of virial identity and variational characterization of two types of standing waves.
%The key point in our approach is to take advantage of 
%We improve the instability results in previous works in one-dimensional case, and moreover establish new instability conditions in higher-dimensional case.

\maketitle

\tableofcontents

\numberwithin{equation}{section} %%
%%%%%%%%%%%%%%%%%%%%%%%%%%%
%%%%%%%%%%%%%%%%%%%%%%%%%%%
\section{Introduction}
\subsection{Setting of the problem}
\label{sec:1.1}
We consider the following nonlinear Schr\"odinger equations:
\begin{align}
\label{eq:NLS}
\tag{NLS}
\l\{
\begin{aligned}
&i\del_t u+\Delta u-a|u|^{p-1}u+b|u|^{q-1}u=0, && (t,x)\in\R\times\R^N,\\%[3pt]
&u(0,x)=u_0(x), &&x\in\R^N,
\end{aligned}
\r.
\end{align}
where $a, b\in\R$ and
%\begin{align*}
%1<p<q<1+\frac{4}{N-2}\quad \l( 1<p<q<\infty ~\text{if}~N=1\r).
%\end{align*}
\begin{align*}
1<p<q<2^*-1, \quad
2^*:=
\l\{
\begin{aligned}
&\infty &&\text{if}~N=1, 2,\\ %[7pt]
&\ds\frac{2N}{N-2} &&\text{if}~N\geq 3. 
\end{aligned}
\r.
\end{align*}
The Cauchy problem for \eqref{eq:NLS} has been extensively studied (see, e.g., \cite{C} and the references therein).
It is well known that for any $u_0\in H^1(\R^N)$ there exists a unique maximal solution 
\begin{align*}
u\in C((-T_{\rm min}, T_{\rm max}),H^1(\R^N))\cap C^1((-T_{\rm min}, T_{\rm max}),H^{-1}(\R^N)) \quad
%I_{\max} =(-T_{\rm min}, T_{\rm max})
\end{align*}
of \eqref{eq:NLS}, and that the energy and the charge are conserved by the flow:
\begin{align*}
E(u(t))&:=\frac{1}{2}\|\nabla u(t) \|_{L^2}^2 +\frac{a}{p+1}\| u(t)\|_{L^{p+1}}^{p+1} -\frac{b}{q+1}
\| u(t)\|_{L^{q+1}}^{q+1} =E(u_0),\\
\| u(t)\|_{L^2}^2&=\| u_0\|_{L^2}^2
\end{align*}
for all $t\in (-T_{\rm min}, T_{\rm max})$. Moreover, the following blowup criterion holds: if $T_{\max}<\infty$ (resp.\ $T_{\min}<\infty$), then $\|\nabla u(t)\|_{L^2}\to\infty$ as $t\uparrow T_{\max}$ (resp.\ as $t\downarrow -T_{\min}$). 
%%%%%%%%%%%%%%%%%%%%%%%
In \cite{TVZ07} the authors studied typical global properties like global well-posedness, scattering, and blowup for \eqref{eq:NLS}. 

In this paper we study the stability properties of standing waves of \eqref{eq:NLS}. We are especially interested in the case $a>0$ and $b>0$ because the equation for this case has algebraically decaying standing waves with zero frequency, which we call \textit{algebraic standing waves}, as well as usual standing waves decaying exponentially with positive frequency. Our main results concern strong instability for all frequencies when $q\ge 1+4/N$ (Theorem \ref{thm:1.8}) and instability for small frequencies when $q<1+4/N$ (Theorem \ref{thm:1.12}), which give the first results on stability properties of algebraic standing waves. The new instability condition in Theorem \ref{thm:1.12} not only improves previous results in one-dimensional case but also gives a first result on instability in higher-dimensional case. We note that this condition is characterized in terms of algebraic standing waves.

%%%%%%%%%%%%%
%%%%%%%%%%%%%

\subsection{Ground states on zero mass case}
\label{sec:1.2}
%Here and hereafter
We now focus on the equation \eqref{eq:NLS} for the case $a>0$ and $b>0$.
By scalar multiplication and the scaling $u\mapsto\mu\lambda^{\frac{2}{p-1}}u(\lambda^2t ,\lambda x)$, we may always take $a=1$ and $b=1$ in \eqref{eq:NLS}:
%Hence, from now on, we study the equation
\begin{align}
\label{NLS}
\l\{
\begin{aligned}
&i\del_t u+\Delta u-|u|^{p-1}u+|u|^{q-1}u=0,&& (t,x)\in\R\times\R^N, \\
&u(0,x)=u_0(x), && x\in\R^N.
\end{aligned}\r.
%\quad a>0.
\end{align}
%The equation \eqref{NLS} has two types of standing waves.
If we consider the standing wave solution
$
e^{i\omega t}\phi_{\omega}(x),
$
then $\phi_{\omega}$ satisfies the following elliptic equation:
\begin{align}
\label{EL}
-\Delta \phi +\omega\phi +|\phi|^{p-1}\phi -|\phi|^{q-1}\phi =0,\quad x\in\R^N.
\end{align}
%%It is known that \eqref{EL} has positive radial solutions when $\omega\ge0$.
%When $\omega>0$, it is known that the solution of \eqref{EL} decays exponentially.
%On the other hand, when $\omega=0$, the solution has only algebraic decay.
It is known that there exist ground states of \eqref{EL} when $\omega\geq 0$ (see \cite{BeL83}). 
The case $\omega=0$ corresponds to \textit{zero mass} case in elliptic equations, and associated problems are more delicate in many cases compared with the case $\omega >0$.

The ground state of \eqref{EL} for $\omega =0$ appears as an optimizer of 
the following Gagliardo--Nirenberg inequality
\begin{align}
\label{GN}
\| f \|_{L^{q+1}} \leq C_{GN}\| f\|_{L^{p+1}}^{1-\theta}\| \nabla f\|_{L^2}^{\theta}
\quad \text{for all}~f\in \dot{H}^1(\R^N)\cap L^{p+1}(\R^N),
\end{align}
where $(p,q,N)$ and $\theta \in (0,1)$ satisfy 
\begin{align*}
\frac{N}{q+1}=(1-\theta)\frac{N}{p+1} +\theta\l(-1+\frac{N}{2}\r).
\end{align*}
When $q=2p-1$ the optimal constant and the optimizers are explicitly found in \cite{DD02}.
%as in the case of the Sobolev inequality (see \eqref{eq:1.4} below). 
The existence of optimizers in general case was studied in \cite{BFV14}.
We note that the end point case $\theta =1$ in \eqref{GN} turns into the Sobolev inequality
\begin{align}
\label{eq:1.4}
\| f\|_{L^{2^*}} \leq C_S\| \nabla f\|_{L^2}
\quad \text{for all}~f\in \dot{H}^1(\R^N), ~N\geq 3.
\end{align}
The existence and uniqueness of optimizers for \eqref{eq:1.4} were proven in \cite{L83a}, and the optimizer is given by the Talenti function (see \cite{T76})
\begin{align*}
W(x) :=\l( 1+\frac{|x|^2}{N(N-2)}\r)^{-\frac{N-2}{2}}, 
\end{align*} 
which solves the following elliptic equation:
\begin{align}
\label{eq:1.5}
-\Delta W -|W|^{\frac{4}{N-2}}W =0 , \quad x\in\R^N.
\end{align}
%The Talenti function is the stationary solution
The equation \eqref{eq:1.5} corresponds to the stationary problem of \eqref{eq:NLS} for the case $a=0$, $b=1$, and $q=2^*-1$, whose study is important to understand the global dynamics for energy critical evolution equations (see, e.g., \cite{KM06, KV10, DJKM17} and references therein).

%The massless problem of the equation \eqref{EL} has been well studied in terms of nonlinear elliptic equations,
%but not so in terms of nonlinear dispersive equations.
The zero mass problem also arises when we study the solitons for the derivative nonlinear Schr\"odinger equation (DNLS):
\begin{align}
\tag{DNLS}
\label{DNLS}
i\del_t u +\del_x^2 u+ i\del_x(|u|^2 u) =0,\quad (t,x)\in \R\times\R, 
\end{align}
which appears in a model for the propagation of Alfv\'{e}n waves in magnetized plasma.
It is known that \eqref{DNLS} has a two-parameter family of solitons
\begin{align*}
e^{i\omega t}\phi_{\omega ,c}(x-ct)\quad\text{with}\ -2\sqrt{\omega}<c\leq 2\sqrt{\omega},
\end{align*}
where $\phi_{\omega ,c}$ is the complex-valued function,\footnote{The profile $\phi_{\omega ,c}$ is expressed in terms of $\Phi_{\omega ,c}$ as 
\[
\phi_{\omega,c}(x) 
= \Phi_{\omega,c}(x) \exp\l( \frac{i}{2}cx - \frac{3}{4}i \int_{-\infty}^{x} \l|\Phi_{\omega,c}(y)\r|^2 dy \r).
\]
} 
and the modulus of the soliton $\Phi_{\omega ,c}=|\phi_{\omega ,c}|$ satisfies the following elliptic equation:
\begin{align}
\label{eq:1.6}
-\Phi'' +\l( \omega -\frac{c^2}{4} \r)\Phi +\frac{c}{2}|\Phi|^2\Phi -\frac{3}{16}|\Phi|^4\Phi =0, \quad x\in\R.
\end{align}
When $0<c\leq 2\sqrt{\omega}$ this equation corresponds to the equation \eqref{EL} with $p=3$ and $q=5$ in one space dimension. Two types of solitons have the following explicit decay at infinity; 
\begin{align*}
\begin{aligned}
&\text{if}~\omega >c^2/4, &&|\phi_{\omega ,c}(x)| \sim e^{-\sqrt{4\omega -c^2}|x|}\\
&\text{if}~c=2\sqrt{\omega}, &&|\phi_{\omega ,c}(x)| \sim (c|x|)^{-1}
\end{aligned}\quad
\text{as}~ |x| \gg 1,
\end{align*}
and the case $c=2\sqrt{\omega}$ corresponds to algebraic solitons. 
It has been clear in the recent studies of \eqref{DNLS} (see \cite{W15, FHI17, H19a}) that algebraic solitons give a certain threshold for global properties of the solutions.
%play an important role at understanding the global properties of the solution of \eqref{DNLS}.

The aim of this paper is to investigate the stability properties of standing waves of \eqref{NLS}, especially algebraic standing waves (zero mass case), which seem to have been less studied from the viewpoint of nonlinear dispersive equations.
% have not been studied at all

%To this aim 
We first organize the properties of the ground states of \eqref{EL}. The action functional with respect to \eqref{EL} is defined by
\begin{align*}
S_{\omega}(v) :=\frac{1}{2}\|\nabla v\|_{L^2}^2 +\frac{\omega}{2}\| v\|_{L^2}^2 +\frac{1}{p+1}\| v\|_{L^{p+1}}^{p+1} -\frac{1}{q+1}\| v\|_{L^{q+1}}^{q+1}.
\end{align*}
We note that $S_0$ is well defined on $\dot{H}^1(\R^N)\cap L^{p+1}(\R^N)$, and that \eqref{EL} is rewritten as $S_{\omega}'(\phi)=0$.
We set 
\begin{align*}
\scA_{\omega}&:=\{ v\in H^1(\R^N)\setminus\{ 0\} : S_{\omega}'(v)=0\}\quad\text{for}~\omega >0,\\
\scA_{0}&:=\{ v\in \dot{H}^1(\R^N)\cap L^{p+1}(\R^N)\setminus\{ 0\} : S_{0}'(v)=0\},
\end{align*}
and for $\omega \geq 0$ we set
\begin{align*}
d(\omega)&:=\inf\{S_\omega(v):v\in\scA_\omega\}, \\
\scG_{\omega}&:=\{ v\in \scA_{\omega} : S_{\omega}(v)=d(\omega)\}.
\end{align*}
The element of $\scG_{\omega}$ is called a ground state of \eqref{EL}, which is characterized as follows. 
%The ground states of \eqref{EL} are characterized as follows.
\begin{proposition}
\label{prop:1.1}
Let $\omega\geq 0$. There exists a unique, positive, radial, and decreasing function $\phi_{\omega}\in\scA_{\omega}$ such that
\begin{align*}
\scG_{\omega} =\{ e^{i\theta}\phi_{\omega}(\cdot -y): \theta\in\R,~y\in\R^N\}.
\end{align*} 
In particular, $d(\omega)=S_\omega(\phi_\omega)$.
%Note that $\phi_{\omega}$ is uniquely determined for each $\omega\geq 0$.
\end{proposition}
Existence of ground states was first proven in \cite{BeL83}. Symmetry of ground states follows from the result of Ni and Li \cite{LN93} (see also \cite{GNN81}). Uniqueness of positive radial solutions follows from the results of Serrin and Tang \cite{ST00} for $N\geq 3$ and the results of Pucci and Serrin \cite{PS98} for $N=2$. In Section \ref{sec:2} we revisit the existence theory on the ground states of \eqref{EL} and give a variational characterization of them on the Nehari manifold, which is very useful for the proofs of our main theorems.

%For $\phi_{\omega}$ in Proposition \ref{prop:1.1} we write $\phi_\omega(x)=\phi_\omega(r)$, where $r=|x|$.
Next we state the decay estimates of ground states on the zero mass case.
It is well known that the ground states for $\omega >0$ decay exponentially at infinity (see, e.g., \cite{BeL83}).
On the other hand ground states on the zero mass case have slower algebraic decay at infinity. The sharp decay estimates on the zero mass case was first obtained by Ver\'on \cite{V81} (see also \cite{DS10, DSW11}). To state the result we define the exponent $p^*$ as
\begin{align*}
p^* =p^*(N)=
\l\{
\begin{aligned}
&\infty &&\text{if}~N=1, 2,\\ %[7pt]
&\ds\frac{N}{N-2} &&\text{if}~N\geq 3. 
\end{aligned}
\r.
\end{align*}
Then we have the following result:
\begin{proposition}[{\cite{V81}}]\label{prop:1.2}
Let $\phi_{0}$ as in Proposition \ref{prop:1.1}. We write $\phi_0(x)=\phi_0(r)$, where $r=|x|$. Then we have
\begin{align}
\label{eq:1.7}
\phi_0(r) \underset{r\to\infty}{\sim}
\l\{
\begin{aligned}
&r^{-\frac{2}{p-1}}&&\text{if}~p<p^*,
\\
&r^{-(N-2)} \l(\log r \r)^{-\frac{N-2}{2}}&&\text{if}~p=p^*,
\\
&r^{-(N-2)}&&\text{if}~p>p^* ,
\end{aligned}
\r.
\end{align}
where $\phi_0(r)\underset{r\to\infty}{\sim} r^{-a}$ for some $a>0$ means that the limit $\lim_{r\to\infty}\phi_0(r)r^a$ exists and its value is positive. 
%In particular, if $p<1+4/N$, then $\phi_0\in L^2(\R^N)$.
\end{proposition}
%%%%%%%%%%%%%%%
%In one-dimensional case one can easily prove by a simple quadrature method that the decay estimate \eqref{eq:1.7} is sharp. When $q=2p-1$ in higher dimensions, we see that the estimate is sharp from the explicit formula of $\phi_0$ (see \cite{DD02}). The decay estimate in general case is also expected to be sharp but we do not pursue this issue further here. 
%%%%%
\begin{remark}
\label{rem:1.3}
We note the following relation:
\begin{align*}
\frac{2}{p^*-1}=N-2,~
\max \l\{ \frac{2}{p-1}, N-2 \r\}=
\l\{
\begin{aligned}
&\frac{2}{p-1} && \text{if}~p<p^*,\\
&N-2 &&\text{if}~p>p^*.
\end{aligned}
\r.
\end{align*}
\end{remark}
As a simple application of Proposition \ref{prop:1.2} we obtain a necessary and sufficient condition for $\phi_0\in L^2(\R^N)$.
\begin{corollary}
\label{cor:1.4}
$\phi_0\in L^2(\R^N)$ holds if and only if
\begin{align}
\label{eq:1.8}
\begin{aligned}
&1<p<1+\frac{4}{N}&&\text{if}~1\le N\le 3,\\
&1<p\le 1+\frac{4}{N}& &\text{if}~N=4,\\
%&1<p<q& &\text{if}~N\ge 5
&N\ge 5.&&
\end{aligned}
\end{align}
\end{corollary}
%%%
\begin{proof}
If we set $p_c(N):=1+4/N$, we have 
\begin{align}
\label{eq:1.9}
\begin{aligned}
&p_c(3)<p^*(3),~p_c(4)=p^*(4),\\
&p^*(N)<p_c(N)\quad\text{if}~N\ge 5.
\end{aligned}
\end{align}
%%%%
The result follows from Proposition \ref{prop:1.2}, \eqref{eq:1.9} and the following fact
\begin{align*}
\braket{x}^{-\frac{2}{p-1}} \in L^2(\R^N) &\iff p<1+\frac{4}{N},
\\
\braket{x}^{-(N-2)} \in L^2(\R^N) &\iff N>4,
\end{align*}
where $\braket{x}:=(1+|x|^2)^{1/2}$.
\end{proof}
%%%%%%%%%%%%%%%
%%%%%%%%%%%%%%%

\subsection{Stability results in previous works}
%Stability properties of standing waves
\label{sec:1.3}
Here and hereafter we only consider standing waves $e^{i\omega t}\phi_\omega$ whose profile $\phi_\omega$ is a positive and radial ground state as in Proposition~\ref{prop:1.1}.
We note that $\phi_0$ is the $H^1$-solution of \eqref{NLS} if \eqref{eq:1.8} holds. First we give a definition of stability, instability, and strong instability of standing waves.
\begin{definition}\label{def:1.5}
Let $\omega \ge 0$ and further assume \eqref{eq:1.8} if $\omega=0$.
We say that the standing wave $e^{i\omega t}\phi_{\omega}$ of \eqref{NLS} is (orbitally) \emph{stable} if for any $\eps>0$ there exists $\delta>0$ such that the following statement holds:
If $u_0\in H^1(\R^N)$ satisfies $\|u_0-\phi_{\omega}\|_{H^1}<\delta$, 
then the solution $u(t)$ of \eqref{NLS} exists globally in time and satisfies 
\[
\sup_{t\in\R}\inf_{(\theta,y)\in\R\x\R^N}\|u(t)-e^{i\theta}\phi_{\omega}(\cdot-y)\|_{H^1}<\eps.
\]
Otherwise, we say that it is (orbitally) \emph{unstable}.

We say that the standing wave $e^{i\omega t}\phi_{\omega}$ is \emph{strongly unstable} if for any $\eps>0$ there exists $u_0\in H^1(\R^N)$ such that $\|u_0-\phi_{\omega}\|_{H^1}<\eps$, and the solution $u(t)$ of \eqref{NLS} blows up in finite time.
\end{definition}
%%%%%

In the case of the pure power nonlinearity (\eqref{eq:NLS} for $a=0$ and $b>0$), stability properties of standing waves are well understood. When $q> 1+4/N$, Berestycki and Cazenave~\cite{BC81} proved that the standing wave $e^{i\omega t}\psi_{\omega}$\footnote{For clarity we write the profile of standing waves as $\psi_{\omega}$ for the pure power nonlinearity.} is strongly unstable for all $\omega>0$ 
(see \cite{W82} for the case $q=1+4/N$), where variational characterizations related to the virial identity are effectively used. 
When $q<1+4/N$, Cazenave and Lions~\cite{CL82} proved that the standing wave 
$e^{i\omega t}\psi_{\omega}$ is stable for all $\omega>0$ by variational and compactness arguments.
Note that when $a=\omega=0$ and $b>0$, the stationary problem of \eqref{eq:NLS} does not have any nontrivial solution decaying at infinity. 

%The works \cite{BC81, CL82} depend on some sort of variational arguments. 
Weinstein \cite{W85, W86} introduced another approach for stability study, and later this approach was generalized by Grillakis, Shatah, and Strauss \cite{GSS87, GSS90}. Roughly speaking, the abstract theory in \cite{GSS87, GSS90} says that the standing wave $e^{i\omega t}\phi_\omega$ is stable if $\del_\omega\|\phi_\omega\|_{L^2}^2>0$ and unstable if $\del_\omega\|\phi_\omega\|_{L^2}^2<0$, provided some spectral conditions of the linearized operator $S_{\omega}''(\phi_{\omega})$ are satisfied.
For the case of pure power nonlinearity, by the scaling symmetry of the equation, the profiles of standing waves are rewritten as $\psi_{\omega}(x)=\omega^{\frac{1}{q-1}}\psi_1(\sqrt{\omega}\, x)$. From this relation one can easily compute the quantity $\del_\omega\|\psi_\omega\|_{L^2}^2$, and for any $\omega>0$ we have
\begin{align*}
%\begin{aligned}
\del_{\omega} \|\psi_{\omega}\|_{L^2}^2>0 &\iff q<1+\frac{4}{N},\\
\del_{\omega} \|\psi_{\omega}\|_{L^2}^2<0 &\iff q>1+\frac{4}{N}.
%\end{aligned}
\end{align*}
However, if we consider double power nonlinearities, it is very delicate to investigate the sign of $\del_\omega\|\phi_\omega\|_{L^2}^2$, especially in higher dimensions, due to the lack of scaling symmetry of the equation. In the first place it is rather nontrivial in general to check that spectral assumptions are satisfied. We also note that the abstract theory is not applicable to the zero mass case $\omega=0$, even for one-dimensional case, due to the lack of coercivity property of the linearized operator. 
%That is why stability of standing waves $\phi_0$ has not been investigated so far. 
%double power lack of scaling symmetry analysis on the stability more delicate

%Now we state the previous results on the equation \eqref{NLS}. 
In one-dimensional case, Iliev and Kirchev~\cite{IK93} calculated the quantity $\del_\omega\|\phi_\omega\|_{L^2}^2$ for rather general nonlinearities and established the stability results. Ohta~\cite{O95dp} further studied stability properties for the case of double power nonlinearities by using the formula of \cite{IK93}, and proved the following result for \eqref{NLS} when $N=1$:
\begin{itemize}
\item
When $q\ge 5$, then the standing wave $e^{i\omega t}\phi_\omega$ is unstable for all $\omega>0$.
\item
When $q<5$, there exists $\omega_1>0$ such that the standing wave $e^{i\omega t}\phi_\omega$ is stable for $\omega>\omega_1$.
 Assuming further $p+q>6$, then there exists $\omega_0\in (0,\omega_1)$ such that the standing wave $e^{i\omega t}\phi_\omega$ is unstable for $\omega\in (0,\omega_0)$.
\end{itemize}
%There exists a gap between $\omega_1$ and $\omega_2$ 
Later, Maeda~\cite{M08} improved this result and bridged a gap between $\omega_0$ and $\omega_1$.
%\begin{quote}
%if $7/3<p<q<5$ and $p+q>6$, then there exists $\omega^*>0$ such that the standing wave $e^{i\omega t}\phi_\omega$ is unstable for $\omega\in(0,\omega^*]$, and stable for $\omega>\omega^*$.
%\end{quote}
The results in \cite{O95dp, M08} imply that stability properties of standing waves for double power nonlinearities may change for the values of $\omega$, which give quite different phenomena from the case of pure power nonlinearity.
%%%%%%%%%%%%%%%%%
%%%%%%%%%%%%%%%%%

The derivation of the formula in \cite{IK93} heavily depends on one dimension, so one cannot expect that similar calculations hold in higher dimensions. One of useful approaches to study stability properties of standing waves in higher dimensions is perturbation arguments as follows. Let us consider the rescaled function of $\phi_{\omega}$ as
\begin{align}
\label{eq:1.10}
 \phi_{\omega}(x)=\omega^{\frac{1}{q-1}}\til{\phi}_{\omega}(\sqrt{\omega}\,x) 
 \quad\text{for}~\omega >0.
\end{align}
We see that $\til{\phi}_{\omega}$ is the positive and radial solution of the equation
\begin{align}
\label{eq:1.11}
-\Delta\til{\phi}_{\omega}+\til{\phi}_{\omega}+\omega^{-\frac{q-p}{q-1}}|\til{\phi}_{\omega}|^{p-1}\til{\phi}_{\omega}-|\til{\phi}_{\omega}|^{q-1}\til{\phi}_{\omega} =0,\quad x\in\R^N.
\end{align}
The third term formally goes to $0$ as $\omega\to\infty$, so one can expect the stability properties for large $\omega$ are similar as in the case of pure power nonlinearity.
Such arguments have been developed in \cite{FO03a, FO03b, F03, CO14} to study the stability properties of standing waves for several types of nonlinear Schr\"{o}dinger equations. We note, however, that this type of arguments are not valid for small $\omega$ in our setting. Indeed, even if we consider another rescaled function of $\phi_{\omega}$ like vanishing the term $-|\phi|^{q-1}\phi$ as $\omega\downarrow 0$, this does not bring any useful information on stability properties because \eqref{eq:NLS} has no standing waves when $a>0$ and $b=0$.

%To our knowledge, stability properties of algebraic standing waves for \eqref{NLS} have not been studied in the literature despite the existence. In this paper we establish instability results for standing waves of \eqref{NLS} including algebraic standing waves by taking advantage of variational characterization of ground states. 

\subsection{Main results}
\label{sec:1.4}
%As we saw in Section \ref{sec:1.2}, there exist ground states with two different decays in \eqref{NLS}. 
%%
We now state our main results. Our first theorem gives a connection of two types of ground states. This result is of independent interest and also used for the proof of Theorem \ref{thm:1.12} below.
%a result itself is independent interest.
\begin{theorem}
\label{thm:1.6}
Let $\phi_{\omega}$ as in Proposition \ref{prop:1.1}. Then we have $\phi_{\omega}\to \phi_0$ strongly in $\dot{H}^1(\R^N)\cap L^{p+1}(\R^N)$ as $\omega\downarrow 0$. If we further assume \eqref{eq:1.8}, we have $\phi_{\omega}\to \phi_0$ strongly in $H^1(\R^N)$ as $\omega\downarrow 0$.
\end{theorem}
\begin{remark}
\label{rem:1.7}
From the proof of Theorem \ref{thm:1.6} we obtain that $d(\omega)\to d(0)$ as $\omega\downarrow 0$. We note that $\{\|\phi_{\omega}\|_{L^2}\}_{0<\omega <1}$ is unbounded in general, but it is true that $\omega\|\phi_\omega\|_{L^2}^2\to0$ as $\omega\downarrow 0$.
%one can improve the convergence in Theorem \ref{thm:1.6} as 
\end{remark}
%%%
The connection of solitons with two different decays for \eqref{DNLS} has been studied in \cite{H18, H19a} but the proofs depend on the explicit formulae of solitons, which is not applicable to at least higher-dimensional case in our setting. Here we use variational characterization of ground states effectively for the proof of Theorem \ref{thm:1.6}.

We now state the instability results of the standing waves of \eqref{NLS}. First we consider the case $q\ge1+4/N$. In this case we obtain the following strong instability result.
%More precisely, we have the following theorem.

\begin{theorem}[Strong instability]
\label{thm:1.8}
Let $q\ge1+4/N$. Then the following statements hold:
\begin{itemize}
\item
If $\omega>0$, then the standing wave $e^{i\omega t}\phi_\omega$ of \eqref{NLS} is strongly unstable.
\item
If \eqref{eq:1.8} holds, then the standing wave $\phi_0$ of \eqref{NLS} is strongly unstable.
\item
If \eqref{eq:1.8} does not hold, then the standing wave $\phi_0$ of \eqref{NLS} is strongly unstable in the following sense:
For any $\eps>0$ there exists $u_0\in H^1(\R^N)$ such that $\|u_0-\phi_0\|_{\dot H^1\cap L^{p+1}}<\eps$, and the solution $u(t)$ of \eqref{NLS} blows up in finite time.
\end{itemize}
\end{theorem}
\begin{remark}
\label{rem:1.9}
%If \eqref{eq:1.8} does not hold, then 
If $\phi_0\notin L^2(\R^N)$, we need to take the weaker topology of distance between $u_0$ and $\phi_0$ than the other cases. The space $\dot H^1(\R^N)\cap L^{p+1}(\R^N)$ naturally arises in variational problems for \eqref{EL} on the zero mass case; see the definition of the action functional $S_0$.
\end{remark}
\begin{remark}
\label{rem:1.10}
As we remarked in Section \ref{sec:1.2}, the solitons of \eqref{DNLS} are closely related to the standing waves of \eqref{NLS} when $p=3$, $q=5$, and $N=1$. It was proven in \cite{CO06} that the soliton $\phi_{\omega ,c}$ for \eqref{DNLS} is stable if $\omega >c^2/4$. On the other hand the stability/instability of $\phi_{\omega,c}$ for $c=2\sqrt{\omega}$ (zero mass case) has been an open problem.
\end{remark}
%%%
\begin{remark}
\label{rem:1.11}
Cheng, Miao, and Zhao~\cite{CMZ16} obtained partial results of Theorem~\ref{thm:1.8} only in the case of $p=1+4/N<q$ and $\omega>0$.
Here we establish strong instability for all cases of $(p,q,\omega)$ satisfying $p<q$, $q\ge1+4/N$, and $\omega\ge0$ including the zero mass case.
\end{remark}
%%%%%%%%%%%%%%%
For the proof of Theorem \ref{thm:1.8} we use the arguments developed in \cite{L08, OY15}, 
which are improvement of the method introduced in \cite{BC81} and still work for the zero mass case. The important point in this approach is that we do not need to solve new variational problems on a virial type functional; instead we take advantage of the variational characterization on the existence theory of the ground states.

%Positive energy
%One can easily see that $E(\phi_{\omega})>0$ for any $\omega\geq 0$. From this point of view, when $q\geq 1+4/N$, the situation is similar to the case of pure power nonlinearity. 
%So strong instability result is not surprising, although the result itself seems to be new.
%%%%%%%%%%%%%%%%%%%%
%%%%%%%%%%%%%%%%%%%%

Next we consider the case $q<1+4/N$.
%, which is the most interesting case in our setting.
It follows from the standard argument (see, e.g., \cite[Chapter 6]{C}) that for any initial data $u_0 \in H^1(\R^N)$ the corresponding solution of \eqref{NLS} exists globally in time and
$\sup\{ \| u(t)\|_{H^1} : t\in\R  \}<\infty$, so in particular the strongly instability does not occur in this case. Also, by the perturbation arguments as stated in Section \ref{sec:1.3}, one can prove that 
the standing wave $e^{i\omega t}\phi_{\omega}$ is stable for large $\omega$ (see \cite{F03}). Therefore, it is more delicate to establish the instability results in the case $q<1+4/N$ than the case $q\geq 1+4/N$.

In the following theorem we establish a new instability condition. Our result not only improves the previous results in one-dimensional case but also gives a first result on instability for the case $q<1+4/N$ in higher dimensions.
%%%%
\begin{theorem}[Instability for small $\omega$]
\label{thm:1.12}
Let $(p,q)$ satisfy 
\begin{equation}
\label{eq:1.12}
1<p<q<1+\frac{4}{N},\quad\gamma_N(p)<q,
\end{equation}
where $\gamma_N(p)$ is defined by
\begin{align}
\label{eq:1.13}
\gamma_N(p):=\frac{16+N^2+6N-pN(N+2)}{N\bigl(N+2-(N-2)p\bigr)}.
\end{align}
%Then there exists $\omega_1>0$ such that the standing wave $e^{i\omega t}\phi_\omega$ is stable for all $\omega>\omega_1$.
Then there exists $\omega_0>0$ such that the standing wave $e^{i\omega t}\phi_\omega$ of \eqref{NLS} is unstable for all $\omega\in[0,\omega_0]$.
\end{theorem}
%%%%%%%%%%%%%%%%%%%%%%
%%%%%%%%%%%%%%%%%%%%%%
\begin{remark}
\label{rem:1.13}
By elementary calculations the curve $(1,2^*-1)\ni p\mapsto\gamma_N(p)$ has the following properties
(see also Figure~\ref{fig:1}):
%
%% figure 環境 
\begin{figure}
\begin{minipage}[b]{0.32\linewidth}
\includegraphics[width=\linewidth]{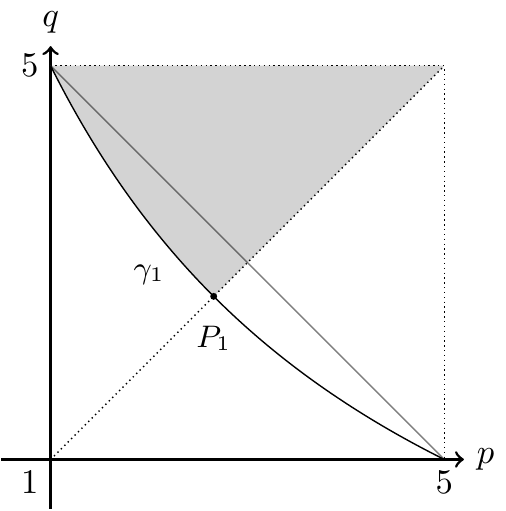}
\subcaption{$N=1$}
\end{minipage}
\begin{minipage}[b]{0.32\linewidth}
\includegraphics[width=\linewidth]{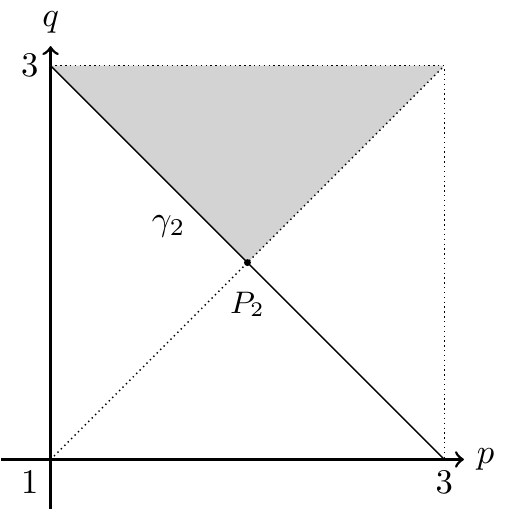}
\subcaption{$N=2$}
%keepaspectratio, scale=1
\end{minipage}
\begin{minipage}[b]{0.32\linewidth}
\includegraphics[width=\linewidth]{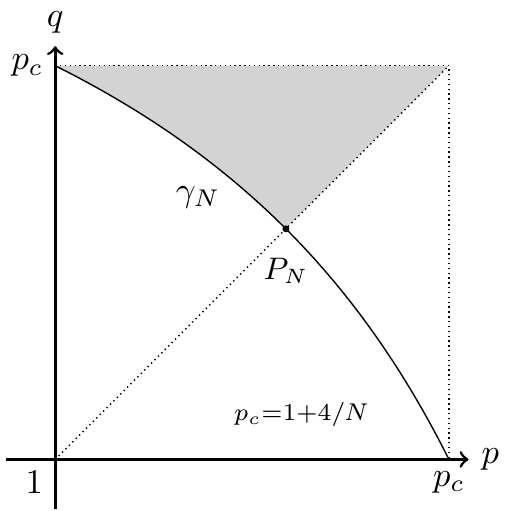}
\subcaption{$N\geq 3$}
\end{minipage}
\caption{The region represented by the condition \eqref{eq:1.12}.}
\label{fig:1}
\end{figure}
%%%%%%%%%%%%%%%%%%
\begin{itemize}
%\item
%$\gamma_N(p)$ is defined on $p\ne2^*-1$.
\item
The curve $q=\gamma_N(p)$ is symmetric about the line $q=p$.
\item
$\gamma_N'(p)=-\frac{32}{N (N+2-(N-2)p)^2}<0$, and so that $\gamma_N$ is strictly decreasing.
\item
$\gamma_N''(p)=-\frac{64(N-2)}{N (N+2-(N-2)p)^3}$.
In particular, the curve $\gamma_N$ is convex if $N=1$, straight if $N=2$, and concave if $N\ge 3$.
\item
$\gamma_N(1)=1+4/N$ and $\gamma_N(1+4/N)=1$.
\item
The curve $q=\gamma_N(p)$ and the line $q=p$ have a unique intersection point $P_N=(p_N,p_N)$, where
\begin{equation}\label{eq:1.14}
p_N:=\frac{N+\sqrt{2N}+4}{\sqrt{N}(\sqrt{N}+\sqrt{2})}
=\frac{\sqrt{N}(N+2)-4\sqrt{2}}{\sqrt{N}(N-2)}.
\end{equation}
\end{itemize}
\end{remark}
%%%%%%%%%%%%%%%%%%%%%%
%%%%%%%%%%%%%%%%%%%%%%
\begin{remark}
\label{rem:1.14}
By using $p_N$ defined in \eqref{eq:1.14}, we can rewrite the condition~\eqref{eq:1.12} as
\begin{align}
\label{eq:1.15}
%\text{if}~
1<p<p_N\ \text{and}\ \gamma_N(p)<q<1+\frac{4}{N},~\text{or}~~p_N\leq p<q<1+\frac{4}{N}.
\end{align}
%in one-dimensional case $N=1$
In particular when $N=1$, \eqref{eq:1.15} is rewritten as
\begin{align}
\label{eq:1.16}
%\text{if}~
1<p<4\sqrt{2}-3\ \text{and}\ \dfrac{23-3p}{3+p}<q<5,
~\text{or}~~4\sqrt{2}-3\le p<q<5.
\end{align}
We note that $\frac{23-3p}{3+p}<6-p$ for all $p\in(1,5)$ (see also Figure~\ref{fig:1} (a)). 
This means that the condition $1<p<q<5$ and $p+q>6$ in \cite{O95dp} is strictly stronger than our condition \eqref{eq:1.16}.
\end{remark}

Let us explain the strategy of proof of Theorem~\ref{thm:1.12}.
First, by following the argument of Ohta~\cite{O95ds}, for $\omega\geq 0$ we establish the following instability result (see Proposition \ref{prop:5.1}):
\begin{align}
\label{eq:1.17}
\bigl.\del_\lambda^2 S_{\omega}(\phi^{\lambda}_{\omega})\bigr|_{\lambda =1}<0
\;\Longrightarrow\; e^{i\omega t}\phi_\omega~\text{is unstable},
\end{align}
where $\phi_{\omega}^{\lambda}(x):=\lambda^{N/2}\phi_{\omega}(\lambda x)$ for $\lambda>0$. The proof of this result is based on the variational characterization of ground states and the virial identity. For the zero mass case we need to modify the argument by using suitable cut-off functions.

For positive $\omega>0$, it is difficult to check the condition in \eqref{eq:1.17}.
On the other hand, for $\omega=0$, one can exactly rewrite the condition in \eqref{eq:1.17} as
\begin{align}
\label{eq:1.18}
\bigl.\del_\lambda^2 S_{0}(\phi^{\lambda}_{0})\bigr|_{\lambda =1} <0\iff \gamma_N(p)<q.
\end{align}
Therefore, \eqref{eq:1.18} implies the instability of algebraic standing waves $\phi_0$ under the condition \eqref{eq:1.12}. Moreover, combining these facts with the convergence result of Theorem~\ref{thm:1.6}, we obtain the instability of standing waves for sufficiently small $\omega$. 

The proof of instability of the standing waves for small $\omega >0$ can be regarded as a certain perturbation argument. However, a major difference from perturbation arguments in previous works is that 
we obtain a new information from algebraic standing waves, whose properties have been less understood, not from well-known standing waves for the pure power nonlinearity.
We emphasize that the new condition \eqref{eq:1.12} is obtained for the first time when we focus on algebraic standing waves.

%%%%%%%%%%%%%%%%%%
\subsection{Outline of the paper}
The rest of this paper is organized as follows.
In Section \ref{sec:2.1} we give a variational characterization of the ground states on the Nehari manifold, and then, in Section \ref{sec:2.2} we complete the proof of Proposition \ref{prop:1.1} by using the results of symmetry and uniqueness.
%where ground states on the cases $\omega>0$ and $\omega=0$ are treated in a unified way. 
% In Section \ref{sec:3.1} we prove the uniform decay estimates of $\{ \phi_{\omega}\}_{0\leq\omega\leq1}$ (see Proposition \ref{prop:3.5}), which is a stronger statement than Proposition \ref{prop:1.2}.
%%%%%
In Section~\ref{sec:3} we study the connection between two types of standing waves and prove Theorem \ref{thm:1.6}. In Section~\ref{sec:4} we study strong instability of standing waves for the case $q\ge1+4/N$ and prove Theorem \ref{thm:1.8}. Finally, in Section~\ref{sec:5} we study instability of standing waves for the case $q<1+4/N$ and prove Theorem \ref{thm:1.12}.
%%%%%%
For the reader's convenience, in Appendix~\ref{sec:A}, we give a radial compactness lemma on the function space $\dot{H}^1(\R^N)\cap L^{p+1}(\R^N)$.
%In Appendix~%ref{sec:B}% we revisit the stability results for sufficiently large frequency by the perturbation arguments.

%%%%%%%%%%%%%%%%%%%%%%%%%%%%%
%%%%%%%%%%%%%%%%%%%%%%%%%%%%%
\section{Properties of ground states}
\label{sec:2}
\subsection{Variational characterization}
\label{sec:2.1}
In this subsection we give a variational characterization of ground states. %We set
%\begin{align*}
%d(\omega) :=\inf \{ S_{\omega}(u) : u\in\scA_{\omega}\}.
%\end{align*}
For convenience, we use the notation
\begin{equation*}
%%label{eq:2.1}%
X=X_{\omega}
:=\left\{
\begin{aligned}
&\dot{H}^1(\R^N)\cap L^{p+1}(\R^N) & & \text{if $\omega=0$},\\
&H^1(\R^N) & & \text{if $\omega>0$}.
\end{aligned}
\right.
\end{equation*}
We define the Nehari functional on $X$ by
\begin{align*}
K_{\omega}(v) :=\|\nabla v\|_{L^2}^2 +\omega\| v\|_{L^2}^2 +\| v\|_{L^{p+1}}^{p+1} -\| v\|_{L^{q+1}}^{q+1}.
\end{align*}
We consider the minimization problem on the Nehari manifold. 
For $\omega\geq 0$ we set
\begin{align*}
\scK_{\omega}&:=\{ v\in X_\omega\setminus\{ 0\} : K_{\omega}(v)=0\}
\end{align*}
and 
\begin{align*}
\mu(\omega)&:=\inf\{ S_{\omega}(v) :v\in\scK_{\omega}\},\\
\notag
\scM_{\omega}&:=\{ v\in\scK_{\omega} :S_{\omega}(v)=\mu (\omega)\}.
\end{align*}
In what follows in this subsection, we prove the following theorem by classical variational arguments.
\begin{theorem}
\label{thm:2.1}
Let $\omega\geq 0$. Then we have 
\begin{align*}
\scG_{\omega}=\scM_{\omega}\neq\emptyset,
\end{align*}
and $d(\omega)=\mu (\omega)>0$.
\end{theorem}
For the proof we prepare two useful lemmas on concentration compactness.
\begin{lemma}[\cite{L83b, BFV14}]
\label{lem:2.2}
Let $p\geq 1$. Let $\{f_n\}$ be a bounded sequence in $\dot{H}^1(\R^N)\cap L^{p+1}(\R^N)$. Assume that there exists $q\in(p,2^*-1)$ such that $\limsup_{n \to \infty} \norm[f_n]_{L^{q+1}}>0$.
Then there exist $\{y_n\}\subset\R^N$ and $f \in \dot{H}^1(\R^N)\cap L^{p+1}(\R^N)\setminus \{0\}$ such that $\{f_n(\cdot-y_n)\}$ has a subsequence that converges to $f$ weakly in $\dot{H}^1(\R^N)\cap L^{p+1}(\R^N)$. 
\end{lemma}
%Br\'ezis--Lieb lemma.
\begin{lemma}[{\cite{BL83}}]
\label{lem:2.3}
Let $1\leq r < \infty$. Let $\{f_n\}$ be a bounded sequence in $L^r(\R^N)$ and $f_n \to f$ a.e.\ in $\R^N$ as $n\to \infty$. Then 
\begin{align*}
\| f_n\|_{L^r}^r - \| f_n-f\|_{L^r}^r - \| f \|_{L^r}^r \to 0
\end{align*}
as $n \to \infty$. 
\end{lemma}

First we show that if $\scM_\omega$ is not empty, then $\scG_\omega$ is characterized as $\scM_\omega$.

\begin{lemma}\label{lem:2.4}
$\scM_\omega\subset\scG_\omega$.
\end{lemma}

\begin{proof}
Let $\phi\in\scM_\omega$. 

First we show that $\phi\in\scA_\omega$.
Using $K_\omega(\phi)=0$, we have
\begin{equation}
\label{eq:2.1}
\begin{aligned}
\langle K_\omega'(\phi),\phi\rangle
&=\bigl.\del_\lambda K_\omega(\lambda\phi)\bigr|_{\lambda=1} \\
&=2(\|\nabla\phi\|_{L^2}^2+\omega\|\phi\|_{L^2}^2)
+(p+1)\|\phi\|_{L^{p+1}}^{p+1}
-(q+1)\|\phi\|_{L^{q+1}}^{q+1} \\
&=-(p-1)(\|\nabla\phi\|_{L^2}^2+\omega\|\phi\|_{L^2}^2)
-(q-p)\|\phi\|_{L^{q+1}}^{q+1}
<0.
\end{aligned}
\end{equation}
%By Lagrange's multiplier theorem, 
Since $\phi\in\scM_\omega$, there exists a Lagrange multiplier $\eta\in\R$ such that $S_\omega'(\phi)=\eta K_\omega'(\phi)$.
From \eqref{eq:2.1} and $\eta\langle K_\omega'(\phi),\phi\rangle=\langle S_\omega'(\phi),\phi\rangle=K_\omega(\phi)=0$, we have $\eta=0$.
This implies that $S_\omega'(\phi)=0$.

Next we show that $\phi\in\scG_\omega$.
Let $v\in\scA_\omega$.
From $K_\omega(v)=\langle S_\omega'(v),v\rangle=0$ and the definition of $\scM_\omega$, we have $S_\omega(\phi)\le S_\omega(v)$,
which yields that $\phi\in\scG_\omega$.
This completes the proof.
\end{proof}

\begin{lemma}\label{lem:2.5}
If $\scM_\omega$ is not empty, then $\scG_\omega\subset\scM_\omega$.
\end{lemma}

\begin{proof}
Let $\phi\in\scG_\omega$.
Since we assume that $\scM_\omega$ is not empty, we can take $\psi\in\scM_\omega$.
By Lemma~\ref{lem:2.4}, we have $\psi\in\scG_\omega$.
Therefore, we obtain $S_\omega(\phi)=S_\omega(\psi)\le S_\omega(v)$ for all $v\in\scK_\omega$, which implies $\phi\in\scM_\omega$.
This completes the proof.
\end{proof}

Next we show that $\scM_\omega$ is not empty.
%Here we prepare some notations. 
We set
\begin{align*}
L_{\omega}(v)&:=\| \nabla v\|_{L^2}^2 +\omega \| v\|_{L^2}^2,
\\
J_{\omega}(v)&:=\l( \frac{1}{2}-\frac{1}{q+1}\r)L_{\omega}(v)+\l( \frac{1}{p+1} -\frac{1}{q+1}\r)\| v\|_{L^{p+1}}^{p+1},
\end{align*}
which are well-defined on $X$.
The functional $S_{\omega}$ is rewritten as
\begin{align}
\label{eq:2.2}
S_{\omega}(v)&=\frac{1}{q+1}K_{\omega}(v)+J_{\omega}(v)
\\
&=\frac{1}{2}K_{\omega}(v)-\frac{p-1}{2(p+1)}\| v\|_{L^{p+1}}^{p+1} +
\frac{q-1}{2(q+1)}\| v\|_{L^{q+1}}^{q+1}.
\label{eq:2.3}
\end{align}
In particular, from \eqref{eq:2.2}, $\mu(\omega)$ is rewritten as
\begin{equation} \label{eq:2.4}
\mu(\omega)
=\inf \{ J_{\omega}(v) : v\in\scK_{\omega}\}.
\end{equation}

\begin{lemma}
\label{lem:2.6}
If $K_{\omega}(v)<0$, then $\mu(\omega) <J_{\omega} (v)$.
In particular, 
\begin{equation} \label{eq:2.5}
\mu(\omega)
=\inf \{ J_{\omega}(v) : v\in X\setminus\{0\},~
K_\omega(v)\le 0\}.
\end{equation}
\end{lemma}

\begin{proof}
Since $K_\omega(v)<0$, we see from the shape of the graph of $\lambda\mapsto K_\omega(\lambda v)$ that there exists $\lambda_1\in(0,1)$ satisfying $K_\omega(\lambda_1v)=0$.
Therefore, by the expression~\eqref{eq:2.4}, we have
\[
\mu(\omega)\le 
J_\omega(\lambda_1v)
<J_\omega(v).
\]
This completes the proof.
\end{proof}

\begin{lemma}\label{lem:2.7}
$\mu(\omega)>0$.
\end{lemma}

\begin{proof}
Let $v\in \scK_\omega$.
By using the Gagliardo--Nirenberg inequality
\[
\|v\|_{L^{q+1}}^{q+1}
\le C_1\|\nabla v\|_{L^{2}}^{q+1}
+C_2\|v\|_{L^{p+1}}^{q+1},
\]
we have 
\begin{equation}
\label{eq:2.6}
0=K_\omega(v)
\ge \left(1-C_1\|\nabla v\|_{L^2}^{q-1}\right)\|\nabla v\|_{L^2}^2
+\left(1-C_2\|v\|_{L^{p+1}}^{q-p}\right)\|v\|_{L^{p+1}}^{p+1}.
\end{equation}
Noting that $v\ne0$, the inequality \eqref{eq:2.6} implies that $\|\nabla v\|_{L^2}\ge (1/C_1)^{1/(q-1)}$ or
$\|v\|_{L^{p+1}}\ge (1/C_2)^{1/(q-p)}$.
Therefore, we obtain
%$J_\omega(v)\gtrsim 1$.
\[
\min\left\{\l(\frac{1}{2}-\frac{1}{q+1}\r)\l(\frac{1}{C_1}\r)^{\frac{2}{q-1}},
\l(\frac{1}{p+1}-\frac{1}{q+1}\r)\l(\frac{a}{C_2}\r)^{\frac{p+1}{q-p}}
\right\}
\le J_\omega(v).
\]
From this and \eqref{eq:2.4}, we have the conclusion.
\end{proof}

Now we prove that $\scM_\omega$ is not empty.

\begin{lemma}
\label{lem:2.8}
If $\{v_n\}\subset X_\omega$ is a minimizing sequence for $\mu(\omega)$, 
that is, 
\[
K_\omega(v_n)\to0,\quad
S_\omega(v_n)\to\mu(\omega),
\]
then there exist $\{y_n\}\subset\R^N$, a subsequence $\{v_{n_j}\}$, and $v_0\in X_\omega\setminus\{0\}$ such that
$v_{n_j}(\cdot-y_{n_j})\to v_0$ in $X_\omega$.
In particular, $v_0\in\scM_\omega$.
\end{lemma}

\begin{proof}
Since $K_\omega(v_n)\to0$ and $S_\omega(v_n)\to\mu(\omega)$, we see from \eqref{eq:2.2} and \eqref{eq:2.3} that
\begin{align}\label{eq:2.7}
J_\omega(v_n)
%=S_\omega(v_n)-\frac{1}{q+1}K_{\omega}(v_n)
&\to\mu(\omega), \\
\label{eq:2.8}
-\frac{p-1}{2(p+1)}\| v_n\|_{L^{p+1}}^{p+1} 
+\frac{q-1}{2(q+1)}\| v_n\|_{L^{q+1}}^{q+1}
%=S_\omega(v_n)-\frac{1}{2}K_{\omega}(u)
&\to\mu(\omega).
\end{align}
It follows from \eqref{eq:2.7} that $\{ v_n\}$ is bounded in $X_{\omega}$. Also, since $\mu(\omega)>0$ by Lemma~\ref{lem:2.7}, it follows from \eqref{eq:2.8} that $\limsup_{n\to\infty}\|v_n\|_{L^{q+1}}>0$. Then, by Lemma~\ref{lem:2.2} there exist $\{y_n\}\subset\R^N$, $v_0\in X_\omega\setminus\{0\}$, and a subsequence of $\{v_n(\cdot-y_n)\}$, which we still denote by the same notation, such that $v_n(\cdot-y_n)\wto v_0$ weakly in $X_\omega$. We put $w_n:=v_n(\cdot-y_n)$.

Now we show the strong convergence of $\{w_n\}$. Taking a subsequence of $\{w_n\}$ if necessary, we may assume that $w_n\to v_0$ a.e.\ in $\mathbb{R}^N$, and that all of limits appearing below exist. By using Lemma~\ref{lem:2.3} we have
\begin{align}\label{eq:2.9}
J_\omega(w_n)
-J_\omega(w_n-v_0)
&\to J_\omega(v_0), \\ \label{eq:2.10}
K_\omega(w_n)
-K_\omega(w_n-v_0)
&\to K_\omega(v_0).
\end{align}
Since $J_\omega(v_0)>0$ by $v_0\ne0$, it follows from \eqref{eq:2.9} and \eqref{eq:2.7} that
\[
\lim_{n\to\infty}J_\omega(w_n-v_0)
=\lim_{n\to\infty}J_\omega(w_n)-J_\omega(v_0)
<\lim_{n\to\infty}J_\omega(w_n)
=\mu(\omega).
\]
From this and \eqref{eq:2.5}, we have $K_\omega(w_n-v_0)>0$ for large $n$. Therefore, by $K_\omega(w_n)\to0$ and \eqref{eq:2.10}, we obtain $K_\omega(v_0)\le0$. Moreover, by \eqref{eq:2.5} and the weakly lower semicontinuity of the norms, we deduce that
\[
\mu(\omega)
\le J_\omega(v_0)
\le \lim_{n\to\infty}J_\omega(w_n)
=\mu(\omega).
\]
This and \eqref{eq:2.9} imply that $J_\omega(w_n-v_0)\to0$, and by the definition of $J_\omega$ we obtain $w_n\to v_0$ strongly in $X_\omega$. This completes the proof.
\end{proof}

\begin{proof}[Proof of Theorem~\ref{thm:2.1}]
The conclusion follows from Lemmas~\ref{lem:2.4}, \ref{lem:2.5}, and \ref{lem:2.8}.
\end{proof}

\subsection{Characterization of ground states}
\label{sec:2.2}
In this subsection we complete the proof of Proposition~\ref{prop:1.1}.
\begin{lemma} \label{lem:2.9}
$\scA_\omega\subset C^2(\R^N,\C)$.
\end{lemma}
\begin{proof}
This follows from the standard bootstrap arguments
(see, e.g., \cite[Theorem~11.7]{LL}).
\end{proof}

\begin{lemma} \label{lem:2.10}
If $\phi\in\scM_\omega$,
then there exists $\theta\in\R$  such that $e^{i\theta}\phi$ is a positive function.
\end{lemma}
%%%
\begin{proof}
When $\omega>0$, we can show the assertion by the same argument as in the proof of \cite[Theorem~8.1.4]{C},
so we omit the proof. We only give a proof in the case of $\omega=0$.

Let $\phi\in\scM_0$.
We put $v :=\lvert\re\phi\rvert$, $w := \lvert\im\phi\rvert$, and $\psi:=v+iw$.
By a phase modulation, we may assume that $v\ne0$.

%First, we show that $\psi\in\scM_0$ and $w=cv$ for some constant $c$.
We note that $|\psi| = |\phi|$ and $|\nabla\psi|=|\nabla\phi|$.
This equalities imply $K_0(\psi)=K_0(\phi)$ and $S_0(\psi)=S_0(\phi)$.
Therefore, Theorem~\ref{thm:2.1} implies $\psi\in\scG_0$.
Since $\psi$ is a solution of \eqref{EL}, $v$ satisfies the equation
\[
(-\Delta+|\phi|^{p-1}-|\phi|^{q-1})v=0.
\]
Since $v$ is nonnegative and not identically equal to zero, we see from \cite[Theorem~9.10]{LL} that $v$ is a positive function.
Moreover, since $K_0(|\psi|)\le K_0(\psi)$ and $S_0(|\psi|)\le S_0(\psi)$,
it follows from Lemma~\ref{lem:2.6} that 
$K_0(|\psi|)=K_0(\psi)$ and $S_0(|\psi|)=S_0(\psi)$,
and so $\|\nabla|\psi|\|_{L^2}=\|\nabla\psi\|_{L^2}$.
Therefore, applying \cite[Theorem~7.8]{LL}\footnote{Although Theorem 7.8 in \cite{LL} is proven for the function of $H^1(\R^N)$, the same proof is valid for the function of $\dot{H}^1(\R^N)\cap L^{p+1}(\R^N)$.} 
%by the same argument as in \cite[Theorem~7.8]{LL}
we deduce that $w=cv$ for some constant $c\geq 0$.

Since $v$ is continuous by Lemma~\ref{lem:2.9} and positive,
$\re\phi$ and $\im\phi$ do not change sign.
This means that there exist constants $\lambda=\pm1$ and $\eta\in\R$ such that $\re\phi=\lambda v$ and $\im\phi=\eta v$.
Taking $\theta\in\R$ so that $e^{-i\theta}=(\lambda+i\eta)/|\lambda+i\eta|$, we obtain $e^{i\theta}\phi=e^{i\theta}(\lambda+i\eta)v=|\lambda+i\eta|v$, which implies that $e^{i\theta}\phi$ is a positive function.
This completes the proof.

\end{proof}

\begin{lemma} \label{lem:2.11}
If $\phi\in\scM_\omega$ is a positive function,
then there exists $y\in\R^N$  such that $\phi(\cdot-y)$ is a radial and decreasing function.
\end{lemma}
%%%
\begin{proof}
This lemma follows from \cite[Theorem~1]{LN93}.
\end{proof}

\begin{proof}[Proof of Proposition~\ref{prop:1.1}]
By Lemmas \ref{lem:2.10} and \ref{lem:2.11}, we can take a positive, radial, and decreasing function $\phi_\omega\in\scM_{\omega}$.
From Theorem~\ref{thm:2.1}, it suffices to show that 
\[
\scM_{\omega}
=\{ e^{i\theta}\phi_{\omega}(\cdot -y): \theta\in\R,\ y\in\R^N\}.
\]
It is trivial that the inclusion $\{ e^{i\theta}\phi_{\omega}(\cdot -y): \theta\in\R, y\in\R^N\}\subset\scM_{\omega}$ holds, and so we only have to verify the inverse inclusion.

Let $\phi\in\scM_\omega$.
By Lemmas~\ref{lem:2.10} and \ref{lem:2.11} there exist $\theta\in\R$ and $y\in\R^N$ such that $e^{-i\theta}\phi(\cdot+y)$ is positive and radial function.
Uniqueness of positive radial solutions follows from Serrin and Tang~\cite[Corollary~(i), Theorem~4]{ST00} for $N\ge3$ and $\omega\ge0$,
%and if $N\ge2$ and $\omega=0$, the uniqueness follows from \cite[Corollary~(i)]{ST00}.
from Pucci and Serrin~\cite[Theorem~2]{PS98} for $N=2$ and $\omega=0$,
and from \cite[Theorem~1]{PS98} for $N=2$ and $\omega>0$.
In any case, we deduce that $e^{-i\theta}\phi(\cdot+y)=\phi_\omega$.
This completes the proof.
%N=2, \omega=0 Pucci--Serrin Theorem 2
\end{proof}
%%%%%%%%%%%%%%%%%%%%
%%%%%%%%%%%%%%%%%%%%

%%%%%%%%%%%%%%%%%%%%%%%%%%%%%%%
%%%%%%%%%%%%%%%%%%%%%%%%%%%%%%%
\section{Connection between two types of standing waves}
\label{sec:3}
%%%%%%%%%%%%%%%%%%%%%
%%%%%%%%%%%%%%%%%%%%%
\subsection{Uniform decay estimates}
\label{sec:3.1}
In this subsection we prove the uniform decay estimate of $\{ \phi_{\omega}\}_{0\le\omega\le1}$. 
We note that $\phi_{\omega}(x)=\phi_{\omega}(r)~(r=|x|)$ satisfies the equation
\begin{align}
\label{eq:3.1}
-\phi_{\omega}'' -\frac{N-1}{r}\phi_{\omega}' +\omega\phi_{\omega}+\phi_{\omega}^p -\phi_{\omega}^q =0,\quad r>0.
\end{align}
%%%%%%
From Theorem \ref{prop:1.1}, Theorem \ref{thm:2.1}, and \eqref{eq:2.2}, for each $\omega\geq 0$ we have
\begin{align}
\label{eq:3.2}
d(\omega)=\mu(\omega)=S_{\omega}(\phi_{\omega})=J_{\omega}(\phi_{\omega})>0. 
\end{align} 
First we show the boundedness of $\{ \phi_{\omega}\}_{0\le\omega\le1}$ by applying variational characterization of $\phi_{\omega}$ in Section \ref{sec:2}. The key for the proof is the following claim:
\begin{lemma}
\label{lem:3.1}
If $0\leq \omega <\omega'$, then $d(\omega)<d(\omega')$.
\end{lemma}
%%%%%%%%%%%%%%
\begin{proof}
For $0\leq \omega<\omega '$ we have
\begin{align*}
K_{\omega}(\phi_{\omega'})=K_{\omega'}(\phi_{\omega'})+(\omega -\omega')
\| \phi_{\omega'}\|_{L^2}^2<0.
\end{align*}
From Lemma \ref{lem:2.6} and $d(\omega)=\mu(\omega)$, we have
\begin{align*}
d(\omega)<J_{\omega}(\phi_{\omega'})<J_{\omega'}(\phi_{\omega'})=d(\omega'),
\end{align*}
which completes the proof.
\end{proof}
%%%%%%%%%%%%%%
\begin{corollary}
\label{cor:3.2}
$\{\phi_{\omega}\}_{0\le\omega\le1}$ is bounded in $\dot{H}^1(\R^N)\cap L^{p+1}(\R^N)$.
\end{corollary}
%%%%%%%%%%%%%%%%%%
\begin{proof}
%Since $d(\omega)=J_{\omega}(\phi_{\omega})$ from \eqref{eq:2.2} and $K_{\omega}(\phi_{\omega})=0$,
From Lemma \ref{lem:3.1} and \eqref{eq:3.2} we have
\begin{align*}
\sup_{0\leq \omega\leq 1}J_{\omega}(\phi_{\omega})\leq d(1).
\end{align*}
Hence the conclusion follows from the explicit formula of $J_{\omega}$.  
\end{proof}
%%%%%%%%%%%%%%%%%
\begin{remark}
\label{rem:3.3}
We note that Lemma \ref{lem:3.1} has been proven without using the smoothness of $\omega\mapsto\phi_{\omega}$ which is a delicate problem in general. 
If we use the differentiability of $\omega\mapsto\phi_{\omega}$, 
it follows from the fact $S_{\omega}'(\phi_{\omega})=0$ that 
%the following relation
\begin{align}
\label{eq:3.3}
\frac{d}{d\omega}d(\omega)=\frac12\|\phi_{\omega}\|_{L^2}^2\quad \text{for}~\omega >0,
%(\because S_{\omega}'(\phi_{\omega})=0),
\end{align}
which yields that $(0,\infty)\ni\omega\mapsto d(\omega)$ is strictly increasing. However \eqref{eq:3.3} does not imply $d(0)<d(\omega)$ for $\omega >0$ because we do not know at this stage whether $\omega\mapsto d(\omega)$ is continuous at $\omega=0$. 
%at $\omega =0$
%%%
%Therefore it would be of independent interest to prove Lemma \ref{lem:3.1} without using the differentiability, even for the case $0<\omega<\omega'$.
\end{remark}
The following elementary inequality is useful to obtain the decay estimate.
\begin{lemma}
\label{lem:3.4}
Let $\varphi\in C^1([0,\infty))$ be a positive function. If there exist $\rho, A>0$ such that
\[
\varphi'(r)+A\varphi(r)^{1+\rho}\leq 0 \quad \text{for all $r>0$},
\]
then
\[
\varphi(r) \leq \l( \frac{1}{\rho Ar}\r)^{1/\rho}\quad \text{for all $r>0$}.
\]
\end{lemma}
%%%%%%%%%%%%%%%%
%%%%%%%%%%%%%%%%
\begin{proposition}
\label{prop:3.5}
Let $\rho:=\max\l\{ \frac{2}{p-1}, N-2\r\}$. Then there exists a constant $C>0$ such that 
\begin{align}
\label{eq:3.4}
\sup_{0\le\omega\le1}\phi_{\omega}(r)\leq Cr^{-\rho}
\quad\text{for large}~r,
\end{align}
where $C$ depends on $p,q,N$.
\end{proposition}
%for some $C$ independent of $p$
\begin{proof}
We first note that
\begin{align*}
\sup_{x\in\R^N}|x|^{\frac{N}{p+1}}|\phi_{\omega}(x)| \leq C\| \phi_{\omega}\|_{L^{p+1}}
\end{align*}
for some $C$ depending on $p$ and $N$, which follows from radial decreasing property of $\phi_{\omega}$ (see the proof of Proposition \ref{prop:A.1}). Combined with Corollary \ref{cor:3.2}, we obtain that
\begin{align}
\label{eq:3.5}
\sup_{0\le\omega\le1}| \phi_{\omega}(x)|\cleq |x|^{-\frac{N}{p+1}} \quad\text{for}~x\in\R^N\setminus\{0\}.
\end{align}
From this estimate, there exists some $r_0>0$, which is independent of $\omega\in[0,1]$, such that
\begin{align}
\label{eq:3.6}
\phi_{\omega}(r)^q \leq \frac{1}{2}\phi_{\omega}(r)^p \quad \text{for}~r\geq r_0.
\end{align}

We now prove the following estimate:
\begin{align}
\label{eq:3.7}
\sup_{0\le\omega\le 1}\phi_{\omega}(r)\leq Cr^{-\frac{2}{p-1}}
\quad\text{for large}~r.
\end{align}
Using the equation \eqref{eq:3.1}, $\phi_{\omega}'(r) <0$, and \eqref{eq:3.6}, we have
\begin{align*}
\phi_{\omega}''\geq \phi_{\omega}'' +\frac{N-1}{r}\phi_{\omega}' =\omega\phi_{\omega}+\phi_{\omega}^p -\phi_{\omega}^q \geq \frac{1}{2}\phi_{\omega}^p \quad\text{for}~r\geq r_0,~\omega\in [0,1].
\end{align*}  
Multiplying the inequality by $\phi_{\omega}'$ and integrating it on $[r, \infty)$, we have
\begin{align*}
\phi_{\omega}'(r)^2 \geq \frac{1}{p+1}\phi_{\omega}(r)^{p+1}
\quad \text{for}\ r\geq r_0,~\omega\in [0,1].
\end{align*}
Since $\phi_{\omega}'(r) <0$ we deduce that
\begin{align*}
\phi_{\omega}'(r) +\sqrt{\frac{1}{p+1}}\, \phi_{\omega}(r)^{\frac{p+1}{2}} \leq 0 
\quad \text{for all $r\geq r_0$,\ $\omega\in [0,1]$}.
\end{align*}
Hence, applying Lemma \ref{lem:3.4}, we obtain \eqref{eq:3.7}.
%Therefore we deduce that
%\begin{align*}
%\sup_{0\leq\omega\leq 1}\phi_{\omega} (r) &\leq \l( \frac{1}{\rho A(r-r_0)}\r)^{1/\rho}
%\leq \l( \frac{2}{\rho A}\r)^{1/\rho} r^{-1/\rho}\quad \text{for all}~r>2r_0.
%\end{align*}

Next we assume $N\ge 3$ and prove the following estimate:
\begin{align}
\label{eq:3.8}
\sup_{0\le\omega\le1}\phi_{\omega}(r)\leq Cr^{-(N-2)}
\quad\text{for large}~r.
\end{align}
Let $\Phi=C|x|^{-(N-2)}$. From \eqref{eq:3.5} we can choose $C>0$, which is independent of $\omega\in [0,1]$, such that
\begin{align*}
\sup_{0\le\omega\le1}\phi_{\omega}(r_0) \le Cr_0^{-(N-2)}=\Phi(r_0).
\end{align*}
From the equation \eqref{EL} and \eqref{eq:3.6} we have
\begin{align*}
\Delta (\phi_{\omega} -C\Phi)=\omega\phi_{\omega}+\phi_{\omega}^p-\phi_{\omega}^{q}\ge\frac{1}{2}\phi_{\omega}^p>0 \quad\text{for}~|x|>r_0,~\omega\in [0,1].
\end{align*}
Here we set $\psi_{\omega}:= \phi_{\omega} -C\Phi$. We note that $\psi_{\omega}$ satisfies
\begin{align*}
\l\{
\begin{aligned}
&\Delta\psi_{\omega} >0&&\text{in}~\l\{x\in\R^N : |x|>r_0 \r\},
\\
&\psi_{\omega} \le 0&&\text{on}~|x|=r_0,
\\
&\psi_{\omega} \to 0&&\text{as}~|x|\to\infty.
\end{aligned}
\r.
\end{align*}
Hence we obtain by the maximum principle that 
\begin{align*}
\sup_{0\le\omega\le1}\psi_{\omega} \le 0\quad \text{for}~|x|\ge r_0, 
%\iff \phi_{\omega} \le Cr^{-(N-2)} \quad \forall r\ge r_0,
\end{align*}
which proves \eqref{eq:3.8}. This completes the proof.
\end{proof}
\begin{remark}
\label{rem:3.6}
The decay estimate \eqref{eq:3.4} is better than \eqref{eq:3.5} because
\begin{align*}
-\frac{2}{p-1}<-\frac{N}{p+1} &\iff (N-2)p<N+2.
\end{align*}
If $N\geq 2$, it follows from Lemma \ref{lem:A.2} and Corollary \ref{cor:3.2} that
\begin{align}
\label{eq:3.9}
\sup_{0\le\omega\le1}|\phi_{\omega}(x)| \cleq |x|^{-\frac{2}{p+3}(N-1)}
\quad\text{for}~x\in\R^N\setminus\{0\}.
\end{align}
Similarly, \eqref{eq:3.4} is better than \eqref{eq:3.9} because
\begin{align*}
-\frac{2}{p-1} < -\frac{2}{p+3}(N-1) \iff (N-2)p<N+2.
\end{align*}
\end{remark}
%%%%%%%%%%%%%%%
%%%%%%%%%%%%%%%
If $p=p^*$, the decay estimate \eqref{eq:3.4} is improved as follows:
\begin{proposition}
\label{prop:3.7}
Assume $N\ge3$ and $p=p^*$. Then there exists a constant $C>0$ such that 
\begin{align}
\label{eq:3.10}
\sup_{0\le\omega\le1}\phi_{\omega}(r)\leq Cr^{-(N-2)}\l(\log r\r)^{-\frac{N-2}{2}} \quad\text{for large}~r,
\end{align}
where $C$ depends on $p,q,N$.
\end{proposition}
%%%
\begin{proof}
We follow the argument in \cite{DSW11}. We set 
\begin{align*}
S(x):=\kappa |x|^{-(N-2)}\l( \log |x|\r)^{-\frac{N-2}{2}},~\text{where}~\kappa:= \l( \frac{N-2}{\sqrt{2}}\r)^{N-2}.
\end{align*}
By a direct calculation we have
\begin{align}
\label{eq:3.11}
\Delta S =S^{p^*}+\frac{N}{2(N-2)}\frac{S^{p^*}}{\log r}.
\end{align}
Fix $R>0$ to be determined later. From \eqref{eq:3.5} we can choose $\eta>1$, which is independent of $\omega\in [0,1]$, such that
\begin{align}
\label{eq:3.12}
\sup_{0\le\omega\le1}\phi_{\omega}(R) \le \eta S(R).
\end{align}
From \eqref{eq:3.11} we have
\begin{align*}
\Delta (\eta S)-(\eta S)^{p^*}=(\eta -\eta^{p^*})S^{p^*}+\frac{\eta N}{2(N-2)}\frac{S^{p^*}}{\log r}.
\end{align*}
For any $\omega\in [0,1]$ we obtain that
\begin{align*}
\Delta (\phi_{\omega}-\eta S)-\frac{\phi_{\omega}^{p^*}-(\eta S)^{p^*}}{\phi_{\omega}-\eta S}\l(\phi_{\omega}-\eta S\r)
&=\Delta \phi_{\omega}-\phi_{\omega}^{p^*} -\l[ \Delta(\eta S)-(\eta S)^{p^*}\r]
\\
&=\omega\phi_{\omega}-\phi_{\omega}^q+(\eta^{p^*}-\eta)S^{p^*}-\frac{\eta N}{2(N-2)}\frac{S^{p^*}}{\log r}
\\
&\ge S^{p^*}\l( (\eta^{p^*}-\eta)- \frac{\eta N}{2(N-2)\log r}-\frac{\phi_{\omega}^q}{S^{p^*}}\r).
\end{align*}
By Proposition \ref{prop:3.5} we obtain that
\begin{align*}
\sup_{0\le\omega\le1}\frac{\phi_{\omega}^q}{S^{p^*}}\cleq \l(\log r\r)^{\frac{p^*(N-2)}{2}}r^{-(q-p^*)(N-2)}\underset{r\to\infty}{\longrightarrow} 0.
\end{align*}
Hence there exists a large $R$, which is independent of $\omega\in [0,1]$, such that
\begin{align}
\label{eq:3.13}
\Delta (\phi_{\omega}-\eta S)-\frac{\phi_{\omega}^{p^*}-(\eta S)^{p^*}}{\phi_{\omega}-\eta S}\l(\phi_{\omega}-\eta S\r)\ge S^{p^*}(\eta^{p^*}-2\eta-1) >0 \quad\text{for}~|x|> R,
%\omega\in [0,1].
\end{align}
where we take $\eta$ large enough so that $\eta^{p^*}-2\eta-1>0$.
Now we set 
\begin{align*}
\Psi_{\omega}:= \phi_{\omega} -\eta S,~c:=\frac{\phi_{\omega}^{p^*}-(\eta S)^{p^*}}{\phi_{\omega}-\eta S}.
\end{align*}
From the monotonicity of $s\mapsto s^{p^*}$, we note that $c>0$.
It follows from \eqref{eq:3.12} and \eqref{eq:3.13} that $\Psi_{\omega}$ satisfies
\begin{align*}
\l\{
\begin{aligned}
&\Delta\Psi_{\omega}-c\Psi_{\omega}>0&&\text{in}~\l\{x\in\R^N : |x|>R \r\},
\\
&\Psi_{\omega} \le 0&&\text{on}~|x|=R,
\\
&\Psi_{\omega} \to 0&&\text{as}~|x|\to\infty.
\end{aligned}
\r.
\end{align*}
Hence we obtain by the maximum principle that 
\begin{align*}
\sup_{0\le\omega\le1}\Psi_{\omega} \le 0\quad\text{for}~|x|\ge R,
\end{align*}
which proves \eqref{eq:3.10}. 
\end{proof}
%%%%%%%%%%%%%%%%%%%%%%
%%%%%%%%%%%%%%%%%%%%%%
\subsection{Zero mass limit}
\label{sec:3.2}
%Asymptotic behavior of $\phi_{\omega}$ as $\omega\downarrow 0$
%In this subsection 
We now complete the proof of Theorem \ref{thm:1.6}.
%%%
\begin{proof}[Proof of Theorem \ref{thm:1.6}]
%%%
We proceed in four steps.\\
%%%%%%%%%%%%%%%%%%%%%%%%%
{\bf Step 1.} $\inf_{\omega >0}\| \phi_{\omega}\|_{L^{q+1}}^{q+1}>0$.
From Lemma \ref{lem:3.1} we have
\begin{align}
\label{eq:3.14}
0<d(0)\leq \inf_{\omega >0}d(\omega) .
\end{align}
From \eqref{eq:2.3} and $K_{\omega}(\phi_{\omega})=0$ we note that
\begin{align}
\label{eq:3.15}
d(\omega)=-\frac{p-1}{2(p+1)}\| \phi_{\omega}\|_{L^{p+1}}^{p+1} +
\frac{q-1}{2(q+1)}\| \phi_{\omega}\|_{L^{q+1}}^{q+1}
\quad\text{for}~\omega\geq 0.
\end{align}
From \eqref{eq:3.14} and \eqref{eq:3.15} we deduce that $\inf_{\omega >0}\| \phi_{\omega}\|_{L^{q+1}}^{q+1}>0$.
\\
{\bf Step 2.} Limits. From Corollary \ref{cor:3.2}, there exists a radial function $\psi\in\dot{H}^1(\R^N)\cap L^{p+1}(\R^N)$ and a sequence $\{\omega_j\}\subset(0,1)$
%\footnote{$\R^+$: the set of positive real numbers.} 
converging to $0$ such that 
\begin{align*}
%%label{eq:3.3}%
\phi_{\omega_j} \wto\psi ~\text{weakly in}~\dot{H}^1(\R^N)\cap L^{p+1}(\R^N) ~\text{as}~j\to\infty.
%\longrightharpoonup
\end{align*}
%In what follows, we shall often extract subsequences of $\{ \omega_j\}$ without mentioning this fact explicitly.
Applying radial compactness lemma (see Appendix \ref{sec:A}), we have
\begin{align}
\label{eq:3.16}
\phi_{\omega_j} \to\psi ~\text{strongly in}~L^{r}(\R^N)~\text{for any}~r\in (p+1, 2^*).
%\longrightharpoonup
\end{align}
It follows from Step 1 and \eqref{eq:3.16} that $\psi\neq 0$. One can easily prove that $\psi$ is a positive radial solution of \eqref{EL}. By uniqueness of positive radial solutions to \eqref{EL} we deduce that $\psi=\phi_0$.
%From the uniqueness of positive radial solutions to
\\
{\bf Step 3.} Strong convergence. 
From \eqref{eq:3.16}, taking a subsequence of $\{ \omega_j\}$ (still denoted by the same letter), we have $\phi_{\omega_j}\to\psi$ a.e.\ in $\R^N$. 
%Applying the Br\'ezis--Lieb lemma
Applying Lemma \ref{lem:2.3}, we have
\begin{align}
\label{eq:3.17}
&K_0(\phi_{\omega_j})-K_{0}(\phi_{\omega_j}-\phi_0)-K_0(\phi_0) \to 0.
\end{align} 
Here we note that
\begin{align}
\label{eq:3.18}
K_{0}(\phi_{\omega}) =K_{\omega}(\phi_{\omega})-\omega\| \phi_{\omega}\|_{L^2}^2 =
-\omega\| \phi_{\omega}\|_{L^2}^2<0 \quad\text{for}~\omega >0,
\end{align}
which yields that $\limsup_{j\to\infty}K_0(\phi_{\omega_j})\le 0$. 
Combined with \eqref{eq:3.17} and $K_0(\phi_0) =0$, we obtain that
\begin{align}
\label{eq:3.19}
\limsup_{j\to\infty}K_0 (\phi_{\omega_j}-\phi_0) \le 0.
\end{align}
By \eqref{eq:3.16}, \eqref{eq:3.19}, and the following relation
\begin{align*}
\| \nabla (\phi_{\omega_j}-\phi_0)\|_{L^2}^2+\| \phi_{\omega_j}-\phi_0\|_{L^{p+1}}^{p+1} 
= K_0 (\phi_{\omega_j}-\psi)+\| \phi_{\omega_j}-\phi_0\|_{L^{q+1}}^{q+1},
\end{align*}
we deduce that 
\begin{align*}
\phi_{\omega_j} \to\phi_0 ~\text{strongly in}~\dot{H}^1(\R^N)\cap L^{p+1}(\R^N).
\end{align*}
%Because of the uniqueness of \eqref{eq:3.1}, 
We note that this convergence does not depend on a sequence $\{ \omega_j\}$ converging to $0$.
%, i.e., the limit of $\phi_{\omega_j}$ is always $\phi_0$. 
Hence we deduce that
\begin{align}
\label{eq:3.20}
\phi_{\omega}\to\phi_0 ~\text{strongly in}~\dot{H}^1(\R^N)\cap L^{p+1}(\R^N)~\text{as $\omega\downarrow0$}.
\end{align}
%%%%%%%%%%
{\bf Step 4.} Improvement of convergence. We now prove the final remark in Theorem \ref{thm:1.6}. Under the assumption \eqref{eq:1.8} it is enough to prove that
\begin{align}
\label{eq:3.21}
\phi_{\omega}\to\phi_0 ~\text{strongly in}~L^2(\R^N)~\text{as $\omega\downarrow0$}.
\end{align}
%%%
We first consider the case 
\begin{align}
\label{eq:3.22}
p<1+4/N~\text{or}~N\ge 5.
\end{align}
Given $R>0$ to be chosen later, we have
\begin{align}
\label{eq:3.23}
\| \phi_{\omega}-\phi_0 \|_{L^2(\R^N)}^2
&=\| \phi_{\omega}-\phi_0 \|_{L^2(B_R)}^2+\| \phi_{\omega}-\phi_0 \|_{L^2(|x|\geq R)}^2, 
\end{align}
where $B_R=\{ x\in\R^N : |x|<R \}$. From Proposition \ref{prop:3.5} we have
\begin{align*}
%%label{unifo}%
\sup_{0<\omega <1}\| \phi_{\omega}-\phi_0 \|_{L^2(|x|\geq R)}\cleq  R^{-\rho+\frac{N}{2}},
\end{align*}
where $\rho=\max\l\{ \frac{2}{p-1}, N-2\r\}$. We note that $\eqref{eq:3.22}$ implies $-\rho+\frac{N}{2}<0$. Hence, for $\eps >0$ one can take $R$ large such that
\begin{align}
\label{eq:3.24}
\sup_{0<\omega <1}\| \phi_{\omega}-\phi_0 \|_{L^2(|x|\geq R)} <\eps.
\end{align}
%and fix this $R$. 
By H\"older's inequality we have
\begin{align*}
\| \phi_{\omega}-\phi_0 \|_{L^2(B_R)} \leq |B_R|^{\frac{1}{2}-\frac{1}{p+1}}\| \phi_{\omega}-\phi_0\|_{L^{p+1}(B_R)}.
\end{align*}
By \eqref{eq:3.20} the right-hand side goes to $0$ as $\omega\downarrow 0$. Hence, combined with \eqref{eq:3.23} and \eqref{eq:3.24}, we deduce that 
\begin{align*}
\limsup_{\omega\downarrow 0}\| \phi_{\omega}-\phi_0 \|_{L^2(\R^N)}\leq \eps,
\end{align*}
which proves \eqref{eq:3.21}.

For the remaining case $N=4$ and $p=p^*(4)=2$, it follows from Proposition \ref{prop:3.7} that
\begin{align*}
\sup_{0<\omega <1}\| \phi_{\omega}-\phi_0 \|_{L^2(|x|\geq R)}
\cleq \l( \log R\r)^{-\frac{1}{2}}.  
\end{align*}
Therefore, \eqref{eq:3.21} is proven in the same way as above. This completes the proof.
\end{proof}
%%%
%\begin{remark}
%%label{rem:3.8}%
%For $\omega >0$ we have the following decay estimate (see, e.g., \cite{BeL83}):
%\begin{align}
%%label{eq:3.24}%
%\phi_{\omega}(x) \leq C|x|^{-\frac{N-1}{2}}e^{-\frac{\sqrt{\omega}}{2}|x|} \quad\text{for}~ |x|\geq 1.
%\end{align}
%Using the estimate \eqref{eq:3.5} one can take a constant $C$ independent of $\omega\in (0,1)$. This yields that $\omega\| \phi_{\omega}\|_{L^2}^2\to 0$ as $\omega\downarrow 0$. 
%\end{remark}
%%%%%%%%%%%%%%%%%
%%%%%%%%%%%%%%%%%
%%%%%%%%%%%%%%%%%%%%%%%%%%%%%%%
%%%%%%%%%%%%%%%%%%%%%%%%%%%%%%%

\section{Strong instability for the case $q\geq 1+4/N$}
\label{sec:4}

In this section we prove Theorem~\ref{thm:1.8}.
Throughout this section we assume that $1<p<q<2^*-1$ and $q\ge1+4/N$.
%To this aim, we prove the following.

The proofs of blowup or instability of standing waves rely on the virial identity. More precisely, we use the following fact: If the initial data $u_0$ of \eqref{NLS} belongs to the function space
\begin{equation*}
\Sigma
:=\{v\in H^1(\R^N):\|xv\|_{L^2}<\infty\},
\end{equation*}
the corresponding $H^1$-solution $u$ of \eqref{NLS} belongs to $C(I_{\max},\Sigma)$, where 
$I_{\max}:=(-T_{\rm min}, T_{\rm max})$ is a maximal interval of $H^1$-solution in Section \ref{sec:1.1}.
Moreover, the function $t\mapsto\|xu(t)\|_{L^2}^2$ is in $C^2(I_{\rm max})$, and the second derivative is expressed as
\begin{equation}\label{eq:4.1}
\dfrac{d^2}{dt^2}\|xu(t)\|_{L^2}^2
=8P(u(t))
\end{equation}
for all $t\in I_{\max}$, where
\begin{align}
\label{eq:4.2}
P(v)
&:=\|\nabla v\|_{L^2}^2
+\frac{\alpha}{p+1}\|v\|_{L^{p+1}}^{p+1}
-\frac{\beta}{q+1}\|v\|_{L^{q+1}}^{q+1}, \\
\label{eq:4.3}
\alpha
&:=\frac{N(p-1)}{2},\quad
\beta
:=\frac{N(q-1)}{2}.
\end{align}
In what follows we often use the following relation
\begin{equation}\label{eq:4.4}
P(v)
=\bigl.\del_\lambda S_\omega(v^\lambda)\bigr|_{\lambda=1},\quad
\text{where}~v^\lambda(x):=\lambda^{N/2}v(\lambda x).
\end{equation}

\subsection{Blowup}

We define a set $\scB_{\omega}$ by
\begin{equation}\label{eq:4.5}
\scB_\omega
:=\{v\in H^1(\R^N):
S_\omega(v)<\mu(\omega),~P(v)<0\},
\end{equation}
where $\mu(\omega)$ is defined in Section \ref{sec:2.1}.
In this subsection, we prove the following blowup result.
\begin{proposition}\label{prop:4.1}
If $u_0\in\scB_\omega\cap\Sigma$, then the solution $u(t)$ of \eqref{NLS} blows up both forward and backward in finite time.
\end{proposition}
\begin{lemma}
\label{lem:4.2}
Let $v\in H^1(\R^N)$.
When $q=1+4/N$, we further assume that $P(v)\le0$.
Then there exists $\lambda_0>0$ such that $K_\omega(v^{\lambda_0})=0$.
\end{lemma}

\begin{proof}
We note that $K_\omega(v^{\lambda})$ is expressed as
\[
K_\omega(v^\lambda)
=\lambda^2\|\nabla v\|_{L^2}^2
+\omega\|v\|_{L^2}^2
+\lambda^\alpha\|v\|_{L^{p+1}}^{p+1}
-\lambda^\beta\|v\|_{L^{q+1}}^{q+1}.
\]
If $q>1+4/N$, that is, if $\beta>2$, the conclusion follows from the shape of the graph of $\lambda\mapsto K_\omega(v^{\lambda})$.

If $q=1+4/N$, we have the expression
\[
K_\omega(v^\lambda)
=\omega\|v\|_{L^2}^2
+\lambda^\alpha\|v\|_{L^{p+1}}^{p+1}
+\lambda^2\left(\|\nabla v\|_{L^2}^2
-\|v\|_{L^{q+1}}^{q+1}\right).
\]
Since $\alpha<2$, we only have to show that 
$\|\nabla v\|_{L^2}^2-\|v\|_{L^{q+1}}^{q+1}<0$.
We note that 
\[
1-\frac{\beta}{q+1}
=1-\frac{2}{q+1}>0.
\]
From this and $P(v)\le0$, we have
\[
\|\nabla v\|_{L^2}^2-\|v\|_{L^{q+1}}^{q+1}
\le-\frac{\alpha}{p+1}\|v\|_{L^{p+1}}^{p+1}
-\left(1-\frac{\beta}{q+1}\right)\|v\|_{L^{q+1}}^{q+1}
<0.
\]
This completes the proof.
\end{proof}

\begin{lemma}\label{lem:4.3}
If $v\in H^1(\R^N)$ satisfies $v\ne0$ and $P(v)\le 0$,
then 
\[
\frac12P(v)
\le S_\omega(v)-\mu(\omega).
\]
\end{lemma}

\begin{proof}
We define the function 
\begin{align*}
f(\lambda)
&:=S_\omega(v^\lambda)
-\frac{\lambda^2}{2}P(v) \\
&\hphantom{:}=\frac{\omega}{2}\|v\|_{L^2}^2
+\frac{\lambda^2}{2}\left(\|\nabla v\|_{L^2}^2
-P(v)\right)
+\frac{a\lambda^\alpha}{p+1}\|v\|_{L^{p+1}}^{p+1}
-\frac{\lambda^\beta}{q+1}\|v\|_{L^{q+1}}^{q+1}.
\end{align*}
First, we prove that $f(1)=\max_{\lambda>0}f(\lambda)$.
To this aim, we divide two cases.

\noindent{\textbf{Case~1:~$q>1+4/N$.}}
Since $v\ne0$ and $P(v)\le 0$, we have $\|\nabla v\|_{L^2}^2-P(v)>0$. Moreover, by \eqref{eq:4.4} we have $f'(1)=0$. Therefore, by the shape of the graph of the function $f(\lambda)$, we see that $f(1)=\max_{\lambda>0}f(\lambda)$.

\noindent{\textbf{Case~2:~$q=1+4/N$.}}
In this case, we express $f(\lambda)$ as
\begin{align*}
f(\lambda)
&=\frac{\omega}{2}\|v\|_{L^2}^2
+\frac{\lambda^\alpha}{p+1}\|v\|_{L^{p+1}}^{p+1}
+\lambda^2\left(\frac12\|\nabla v\|_{L^2}^2
-\frac12P(v)-\frac{1}{q+1}\|v\|_{L^{q+1}}^{q+1}\right) \\
&=\frac{\omega}{2}\|v\|_{L^2}^2
+\frac{\lambda^\alpha}{p+1}\|v\|_{L^{p+1}}^{p+1}
-\frac{\alpha\lambda^2}{2(p+1)}\|v\|_{L^{p+1}}^{p+1}.
\end{align*}
By $f'(1)=0$ and the shape of the graph of $f(\lambda)$,
we see that $f(1)=\max_{\lambda>0}f(\lambda)$.

Next, we prove the desired inequality.
By Lemma~\ref{lem:4.2}, there exists $\lambda_0>0$ such that $K(v^{\lambda_0})=0$.
Therefore, by the definition of $\mu(\omega)$,
$P(v)\le0$, and $f(1)\ge f(\lambda_0)$, we obtain
\[
\mu(\omega)
\le S_\omega(v^{\lambda_0})
\le S_\omega(v^{\lambda_0})-\frac{\lambda_0^2}{2}P(v)
\le S_\omega(v)-\frac{1}{2}P(v).
\]
This completes the proof.
\end{proof}

\begin{lemma}\label{lem:4.4}
The set $\scB_{\omega}$ is invariant under the flow of \eqref{NLS},
that is, if $u_0\in\scB_{\omega}$, then the solution $u(t)$ of \eqref{NLS} satisfies $u(t)\in\scB_{\omega}$ for all $t\in I_{\max}$.
\end{lemma}

\begin{proof}
Let $u_0\in\scB_{\omega}$.
Since $S_\omega$ is a conserved quantity of \eqref{NLS},
we have 
\begin{equation}\label{eq:4.6}
S_\omega(u(t))=S_\omega(u_0)<\mu(\omega)\quad\text{for all}~t\in I_{\max}.
\end{equation}

Now we show that $P(u(t))<0$ for all $t\in I_{\max}$.
If not, by the continuity of the flow, there exists $t_0\in I_{\max}$ such that $P(u(t_0))=0$.
By Lemma~\ref{lem:4.3} we have $\mu(\omega)\le S_\omega(u(t_0))$,
which contradicts \eqref{eq:4.6}.
This completes the proof.
\end{proof}

\begin{proof}[Proof of Proposition~\ref{prop:4.1}]
Let $u_0\in\scB_{\omega}\cap\Sigma$.
By Lemma~\ref{lem:4.4}, we have $u(t)\in\scB_{\omega}\cap\Sigma$ for all $t\in I_{\max}$.
Then, it follows from the virial identity~\eqref{eq:4.1} and Lemma~\ref{lem:4.3} that
\[
\frac{d^2}{dt^2}\|xu(t)\|_{L^2}^2
=8P(u(t))
\le16\bigl(S_\omega(u(t))-\mu(\omega)\bigr)
=16\bigl(S_\omega(u_0)-\mu(\omega)\bigr)
<0
\]
for all $t\in I_{\max}$,
which implies $|I_{\max}|<\infty$. This completes the proof.
\end{proof}

\subsection{Strong instability}

%In this subsection, 
Now we prove Theorem~\ref{thm:1.8}.
%When $p\ge1+4/N$, we cannot expect $\phi_0\in H^1(\R^N)$. 
%when $p\ge 1+4/N$ and $\omega=0$, we do not know $\phi_0\in H^1(\R^N)$.
Here we prepare the following notations:
\[
Y=Y_{p,\omega}
:=\left\{
\begin{aligned}
&H^1(\R^N)& &\text{if $\omega>0$, or if $\omega=0$ and \eqref{eq:1.8} holds}, \\
&\dot H^1(\R^N)\cap L^{p+1}(\R^N) & & \text{otherwise}.
\end{aligned}
\right.
\]
We define a cutoff function by 
\begin{equation} \label{eq:4.7}
\chi_R(x):=\chi(|x|/R)\quad \text{for}~x\in\R^N,
\end{equation}
where $\chi\in C_\mathrm{c}^\infty[0,\infty)$ is a function such that $\chi(r)=1$ if $0\le r\le1$ and $\chi(r)=0$ if $r\geq2$.

\begin{lemma}\label{lem:4.5}
There exists a function $R\colon(1,\infty)\to(0,\infty)$ such that
$\chi_{R(\lambda)}\lambda\phi_\omega\in\scB_{\omega}\cap\Sigma$ for all $\lambda>1$, and that
$\chi_{R(\lambda)}\lambda\phi_\omega\to\phi_\omega$ in $Y$ as $\lambda\downarrow1$. 
\end{lemma}

\begin{proof}
First, we show that $\lambda\phi_\omega\in\scB_{\omega}$ for all $\lambda>1$.
We note that 
\[
S_\omega(\lambda\phi_\omega)
=\frac{\lambda^2}{2}(\|\nabla\phi_\omega\|_{L^2}^2
+\omega\|\phi_\omega\|_{L^2}^2)
+\frac{\lambda^{p+1}}{p+1}\|\phi_\omega\|_{L^{p+1}}^{p+1}
-\frac{\lambda^{q+1}}{q+1}\|\phi_\omega\|_{L^{q+1}}^{q+1}.
\]
By $\bigl.\del_\lambda S_\omega(\lambda\phi_\omega)\bigr|_{\lambda=1}=0$ and the shape of the graph of $\lambda\mapsto S_\omega(\lambda\phi_\omega)$, we see that $S_\omega(\lambda\phi_\omega)<S_\omega(\phi_\omega)$ for all $\lambda>1$.

Next, we show the conclusion. We note that the set
\[
\{v\in Y:S_\omega(v)<\mu(\omega),~P(v)<0\}
\] 
is open in $Y$.
Since $\chi_Rv\to v$ in $Y$ as $R\to\infty$ for any $v\in Y$,
we see that for any $\lambda>1$ there exists $R(\lambda)>0$ such that 
$\|\chi_{R(\lambda)}\lambda\phi_\omega-\lambda\phi_\omega\|_{Y}<\lambda-1$ and
$\chi_{R(\lambda)}\lambda\phi_\omega\in\scB_{\omega}\cap\Sigma$.
This implies the conclusion.
\end{proof}

\begin{proof}[Proof of Theorem~\ref{thm:1.8}.]
The conclusion follows from Proposition~\ref{prop:4.1} and Lemma~\ref{lem:4.5}.
This completes the proof.
\end{proof}
%%%%%%%%%%%%%%%%%%%%%%%%%%%%%%%
%%%%%%%%%%%%%%%%%%%%%%%%%%%%%%%

\section{Instability for the case $q<1+4/N$}
\label{sec:5}
In this section we prove Theorem~\ref{thm:1.12}.
Throughout this section, we assume $1<p<q<1+4/N$.

\subsection{Sufficient conditions for instability}
\label{sec:5.1}

In this subsection we prove the following by using the similar argument of \cite{O95ds}.

\begin{proposition}\label{prop:5.1}
Let $\omega\ge0$ and assume that
\begin{align}
\label{eq:5.1}
\del_\lambda^2 S_{\omega}(\phi^{\lambda}_{\omega})|_{\lambda =1}<0,
\quad\text{where}~v^{\lambda}(x):=\lambda^{N/2}v(\lambda x).
\end{align}
Then the standing wave $e^{i\omega t}\phi_\omega$ is unstable.
\end{proposition}

We define a tube around the standing wave by
\begin{align*}
\scN_\eps
&:=\{v\in H^1(\R^N):\inf_{(\theta,y)\in\R\x\R^N}\|v-e^{i\theta}\phi_\omega(\cdot-y)\|_{H^1}<\eps\}.
\end{align*}

\begin{lemma}\label{lem:5.2}
Assume \eqref{eq:5.1}.
Then there exist $\eps_1,\delta_1\in(0,1)$ such that the following holds:
For any $v\in\scN_{\eps_1}$ there exists $\Lambda(v)\in(1-\delta_1,1+\delta_1)$ such that
\[
\mu(\omega)
\le S_\omega(v)
+\big(\Lambda(v)-1\big)P(v).
\]
\end{lemma}

\begin{proof}
Since $\del_\lambda^2S_\omega(\phi_\omega^\lambda)|_{\lambda=1}<0$,
by the continuity of the function 
\begin{align*}
(\lambda,v)\mapsto \del_\lambda^2S_\omega(v^\lambda),
\end{align*}
there exist $\eps_1,\delta_1\in(0,1)$ such that $\del_\lambda^2S_\omega(v^\lambda)<0$ for any $\lambda\in(1-\delta_1,1+\delta_1)$ and $v\in\scN_{\eps_1}$.
Moreover, by \eqref{eq:4.4} and Taylor's expansion, we have 
\begin{equation}\label{eq:5.2}
S_\omega(v^\lambda)
\le S_\omega(v)
+(\lambda-1)P(v)
\end{equation}
for any $\lambda\in(1-\delta_1,1+\delta_1)$ and $v\in\scN_{\eps_1}$.

Here we note that $K_\omega(\phi_\omega)=0$ and
\begin{align*}
\bigl.\del_\lambda K_\omega(\phi_\omega^\lambda)\bigr|_{\lambda=1}
&=2\|\nabla\phi_\omega\|_{L^2}^2
+\alpha\|\phi_\omega\|_{L^{p+1}}^{p+1}
-\beta\|\phi_\omega\|_{L^{q+1}}^{q+1} \\
&=-(p-1)\|\nabla\phi_\omega\|_{L^2}^2
-\frac{\beta(q-p)}{q+1}\|\phi_\omega\|_{L^{q+1}}^{q+1} \\
&<0,
\end{align*}
where we used $P(\phi_\omega)=0$ in the second equality.
By the implicit function theorem,
taking $\eps_1$ and $\delta_1$ smaller if necessary,
for any $v\in\scN_{\eps_1}$ there exists $\Lambda(v)\in (1-\delta_1,1+\delta_1)$ such that $\Lambda(\phi_\omega)=1$ and $K_\omega(v^{\Lambda(v)})=0$.
Therefore, by the definition of $\mu(\omega)$ and \eqref{eq:5.2}, we obtain
\[
\mu(\omega)
\le S_\omega(v^{\Lambda(v)})
\le S_\omega(v)
+\bigl(\Lambda(v)-1\bigr)P(v).
\]
This completes the proof.
\end{proof}

For the solution $u(t)$ of \eqref{NLS} with $u_0\in\scN_\eps$,
we define the exit time from the tube $\scN_\eps$ by
\[
T_\eps^\pm(u_0)
:=\inf\{t>0:
u(\pm t)\notin\scN_\eps\}.
\]
We set $I_\eps(u_0):=(-T_\eps^-(u_0),T_\eps^+(u_0))$.

\begin{lemma}\label{lem:5.3}
Assume \eqref{eq:5.1}.
For any $u_0\in\scB_{\omega}\cap\scN_{\eps_1}$, where $\scB_\omega$ is the set defined in \eqref{eq:4.5}, there exists $m=m(u_0)>0$ such that $P(u(t))\le -m$ for all $t\in I_{\eps_1}(u_0)$.
\end{lemma}

\begin{proof}
For $t\in I_{\eps_1}(u_0)$,
since $u(t)\in\scN_{\eps_1}$,
it follows from Lemma~\ref{lem:5.2} that
\[
\mu(\omega)-S_\omega(u_0)
=\mu(\omega)-S_\omega(u(t))
\le-\{1-\Lambda(u(t))\}P(u(t)).
\]
In particular,  
since $\mu(\omega)>S_\omega(u_0)$ by $u_0\in\scB_{\omega}$, 
we have $P(u(t))\ne0$.
By the continuity of the flow and $P(u_0)<0$ we obtain
\[
P(u(t))<0,\quad
1-\Lambda(u(t))>0.
\]
Therefore, we obtain
\[
-P(u(t))
\ge\frac{\mu(\omega)-S_\omega(u_0)}{1-\Lambda(u(t))}
\ge\frac{\mu(\omega)-S_\omega(u_0)}{\delta_1}
=:m(u_0)>0.
\]
This completes the proof.
\end{proof}

\begin{lemma}\label{lem:5.4}
Assume \eqref{eq:5.1}.
Then $|I_{\eps_1}(u_0)|<\infty$ for all $u_0\in\scB_{\omega}\cap\scN_{\eps_1}\cap\Sigma$.
\end{lemma}

\begin{proof}
Let $u(t)$ be the solution of \eqref{NLS} with $u_0\in\scB_{\omega}\cap\scN_{\eps_1}\cap\Sigma$.
By the virial identity~\eqref{eq:4.1} and Lemma~\ref{lem:5.3},
we obtain
\[
\frac{d^2}{dt^2}\|xu(t)\|_{L^2}^2
=8P\big(u(t)\big)
\le -8m
\]
for all $t\in I_{\eps_1}(u_0)$,
which implies $|I_{\eps_1}(u_0)|<\infty$.
This completes the proof.
\end{proof}

\begin{proof}[Proof of Proposition~\ref{prop:5.1}]
First, we claim that $\phi_\omega^\lambda\in\scB_{\omega}\cap\scN_{\eps_1}$ for all $\lambda >1$ close to $1$.
Since $\bigl.\del_\lambda S_\omega(\phi_\omega^\lambda)\bigr|_{\lambda=1}=0$ and 
$\bigl.\del_\lambda^2 S_\omega(\phi_\omega^\lambda)\bigr|_{\lambda=1}<0$,
we see that there exists $\lambda_1>1$ such that $\del_\lambda S_\omega(\phi_\omega^\lambda)<0$ and $S_\omega(\phi_\omega^\lambda)<\mu(\omega)$ for all $\lambda\in(1,\lambda_1)$.
We also see that $P(\phi_\omega^\lambda)=\lambda\del_\lambda S_\omega(\phi_\omega^\lambda)<0$ for all $\lambda\in(1,\lambda_1)$.
Moreover, taking $\lambda_1$ smaller if necessary, we have $\phi_\omega^\lambda\in\scN_{\eps_1}$ for all $\lambda\in(1,\lambda_1)$.

We now prove the conclusion.
Let $\chi_R\in C_\mathrm{c}^\infty(\R^N)$ be the cut off function defined in \eqref{eq:4.7}.
It follows from the similar argument as in the proof of Lemma~\ref{lem:4.5} that there exists $R\colon(1,\infty)\to(0,\infty)$ such that
$\chi_{R(\lambda)}\phi_\omega^\lambda\to\phi_\omega$ in $H^1(\R^N)$ as $\lambda\downarrow1$ and
$\chi_{R(\lambda)}\phi_\omega^\lambda\in\scB_{\omega}\cap\scN_{\eps_1}\cap\Sigma$ for all $\lambda\in(1,\lambda_1)$.
By applying Lemma~\ref{lem:5.4} we have $|I_{\eps_1}(\chi_{R(\lambda)}\phi_\omega^\lambda)|<\infty$ for all $\lambda\in(1,\lambda_1)$.
Hence, %since $\chi_{R(\lambda)}\phi_\omega^\lambda\to\phi_\omega$ as $\lambda\downarrow1$ in $H^1(\R^N)$,
the standing wave $e^{i\omega t}\phi_\omega$ is unstable. This completes the proof.
\end{proof}

%The proof of Proposition~\ref{prop:5.1} is given in Subsection~\ref{sec:5.1} below.

\subsection{Instability for small frequencies}
\label{sec:5.2}

We now complete the proof of Theorem \ref{thm:1.12}. The following claim is a key lemma for the proof.

\begin{lemma}\label{lem:5.5}
Let $1<p<q<1+4/N$.
Then the condition~\eqref{eq:5.1} for $\omega=0$ is equivalent to 
\begin{align}
\label{eq:5.3}
\gamma_N(p)=\frac{16+N^2+6N-pN(N+2)}{N\bigl(N+2-(N-2)p\bigr)}<q.
\end{align}
\end{lemma}
%%%%%%%%%%%
\begin{proof}

%For $\omega\geq 0$ the condition 
%\begin{align}
%\label{eq:5.1}
%\Bigl. \del_\lambda^2 S_{\omega}(\phi^{\lambda}_{\omega})\Bigr|_{\lambda =1} <0 
%\quad\text{where}~\phi_{\omega}^{\lambda}(x):=\lambda^{N/2}\phi_{\omega}(\lambda x) 
%\end{align}
%yields that $\phi_{\omega}$ is unstable (see Fukaya's personal note). Here we establish the range $(p,q)$ satisfying \eqref{eq:5.1} when $\omega =0$.

First we note that
\begin{align}
\label{eq:5.4}
S_0(\phi_0^{\lambda})=\frac{\lambda^2}{2}\|\nabla\phi_0\|_{L^2}^2
+\frac{\lambda^{\alpha}}{p+1}\|\phi_0\|_{L^{p+1}}^{p+1}-\frac{\lambda^{\beta}}{q+1}\| \phi_0\|_{L^{q+1}}^{q+1},
\end{align}
where $\alpha$ and $\beta$ are defined by \eqref{eq:4.3}.
From $1<p<q<1+4/N$, we have $0<\alpha <\beta <2$.
If we differentiate \eqref{eq:5.4} with respect to $\lambda$ twice, we have
\begin{align}
\label{eq:5.5}
\bigl.\del_{\lambda}^2S_0(\phi_0^{\lambda})\bigr|_{\lambda=1}
=\|\nabla\phi_0\|_{L^2}^2
+\frac{\alpha(\alpha -1)}{p+1}\|\phi_0\|_{L^{p+1}}^{p+1}-\frac{\beta(\beta-1)}{q+1}\| \phi_0\|_{L^{q+1}}^{q+1}.
\end{align}
Here we use Pohozaev's identities. We note that 
\begin{align*}
K_0(\phi_0)&=\|\nabla\phi_0\|_{L^2}^2+\| \phi_0\|_{L^{p+1}}^{p+1}-\| \phi_0\|_{L^{q+1}}^{q+1}=0,\\
P(\phi_0)&=\|\nabla\phi_0\|_{L^2}^2+\frac{\alpha}{p+1}\| \phi_0\|_{L^{p+1}}^{p+1}-\frac{\beta}{q+1}\| \phi_0\|_{L^{q+1}}^{q+1}=0.
\end{align*}
From these two relations, we have
\begin{align*}
%&\|\nabla\phi_0\|_{L^2}^2+\frac{\alpha}{p+1}\| \phi_0\|_{L^{p+1}}^{p+1}-\frac{\beta}{q+1}
%\l( \|\nabla\phi_0\|_{L^2}^2 +\| \phi_0\|_{L^{p+1}}^{p+1} \r)=0\\
%\iff&
\l( 1-\frac{\beta}{q+1}\r) \| \nabla\phi_0\|_{L^2}^2&= \l( \frac{\beta}{q+1}-\frac{\alpha}{p+1} \r)\| \phi_0\|_{L^{p+1}}^{p+1},\\
%\end{align*}
%If we eliminate $L^{p+1}$-norm, we have
%\begin{align*}
%&\|\nabla\phi_0\|_{L^2}^2+\frac{\alpha}{p+1}
%\l( \| \phi_0\|_{L^{q+1}}^{q+1}-\|\nabla\phi_0\|_{L^2}^2 \r)
%-\frac{\beta}{q+1}\| \phi_0\|_{L^{q+1}}^{q+1}
%=0\\
%\iff&
\l( 1-\frac{\alpha}{p+1}\r) \| \nabla\phi_0\|_{L^2}^2&= \l( \frac{\beta}{q+1}-\frac{\alpha}{p+1} \r)\| \phi_0\|_{L^{q+1}}^{q+1}.
\end{align*}
By substituting these formulae into \eqref{eq:5.5}, we have
\begin{align}
\label{eq:5.6}
\bigl.\del_{\lambda}^2S_0(\phi_0^{\lambda})\bigr|_{\lambda=1}
=\|\nabla\phi_0\|_{L^2}^2\biggl[&1+\frac{\alpha(\alpha -1)}{p+1}
\l( \frac{\beta}{q+1}-\frac{\alpha}{p+1} \r)^{-1}\l( 1-\frac{\beta}{q+1}\r)\\
&-\frac{\beta(\beta -1)}{q+1}
\l( \frac{\beta}{q+1}-\frac{\alpha}{p+1} \r)^{-1}\l( 1-\frac{\alpha}{p+1}\r) \biggr].
\notag
\end{align}
From this expression and elementary calculations, we see that $\bigl.\del_{\lambda}^2S_0(\phi_0^{\lambda})\bigr|_{\lambda=1}<0$ is equivalent to \eqref{eq:5.3}.
%This is equivalent to \eqref{eq:5.3}.
This completes the proof.
\end{proof}
%%%%%%%%%%%%%%
%%%%%%%%%%%%%%
\begin{proof}[Proof of Theorem~\ref{thm:1.12}]
%Let $p,q$ satisfy $1<p<q<1+4/N$ and \eqref{eq:5.3}.
By Theorem~\ref{thm:1.6} we see that $\omega\mapsto\bigl.\del_\lambda^2S_\omega(\phi_\omega^\lambda)\bigr|_{\lambda=1}$ is continuous at $\omega=0$. Since $\bigl.\del_\lambda^2S_0(\phi_0^\lambda)\bigr|_{\lambda=1}<0$ by the assumption \eqref{eq:1.12} and Lemma~\ref{lem:5.5}, there exists $\omega_0>0$ such that $\bigl.\del_\lambda^2S_\omega(\phi_\omega^\lambda)\bigr|_{\lambda=1}<0$ for all $\omega\in[0,\omega_0]$. Hence we deduce by Proposition~\ref{prop:5.1} that the standing wave $e^{i\omega t}\phi_\omega$ is unstable for all $\omega\in[0,\omega_0]$. This completes the proof.
\end{proof}

%S_0(\phi_0)=\| \nabla\phi_0\|_{L^2}^2
%\l(\frac{1}{2} +\frac{\alpha}{p+1}\r)
%%%%%%%%%%%%%%%%%%

%%%%%%%%%%%%%%%%%%%%%%%%%%%%%%%
%%%%%%%%%%%%%%%%%%%%%%%%%%%%%%%
\appendix
\section{Radial compactness lemma}
\label{sec:A}
We establish the radial compactness lemma in a little more general setting than the the original one by Strauss \cite{S77}.
\begin{proposition}
\label{prop:A.1}
Let $1\leq p<2^*-1$ and let $\{ u_n\}\subset\dot{H}^1(\R^N)\cap L^{p+1}(\R^N)$ be a bounded sequence of radial functions. If $N\geq 2$ or if $u_n(x)$ is a nonincreasing function of $|x|$ for every $n\in\N$, then there exist a subsequence $\{ u_{n_j}\}$ and $u\in\dot{H}^1(\R^N)\cap L^{p+1}(\R^N)$ such that 
$u_{n_j}\to u$ in $L^r(\R^N)$ as $j\to\infty$ for every $r\in(p+1,2^*)$.
\end{proposition}
For the proof of Proposition \ref{prop:A.1}, the following lemma is important. 
\begin{lemma}
\label{lem:A.2}
If $u\in \dot{H}^1(\R^N)\cap L^{p+1}(\R^N)$ is a radial function, then
\begin{align}
\label{eq:A.1}
\sup_{x\in\R^N}|x|^{\frac{2}{p+3}(N-1)}|u(x)| 
\leq \| u\|_{L^{p+1}}^{\frac{p+1}{p+3}} \|\nabla u\|_{L^2}^{\frac{2}{p+3}}.
\end{align}
\end{lemma}
%%%%%
\begin{proof}
Let the exponent $\rho$ to be chosen later. We may assume that $u\in C^{\infty}_c(\R^N,\R)$. Then we have
\begin{align*}
r^{N-1}|u(r)|^{\rho} &=-\int_{r}^{\infty} \frac{d}{ds} ( s^{N-1}|u(s)|^{\rho})\,ds\\
&=-\int_{r}^{\infty}(N-1)s^{N-2}|u(s)|^{\rho}\,ds -\int_{r}^{\infty}s^{N-1}|u(s)|^{\rho -2}u(s)u'(s)\, ds \\
&\le-\int_{r}^{\infty}s^{N-1}|u(s)|^{\rho -2}u(s)u'(s)\, ds \\
%\end{align*}
%By the Cauchy-Schwartz inequality we have
%\begin{align*}
&\leq \l( \int_{r}^{\infty}\bigl(s^{\frac{N-1}{2}} |u(s)|^{\rho -1}\bigr)^2 ds\r)^{1/2} 
\l( \int_{r}^{\infty} |s^{\frac{N-1}{2}} u'|^2 ds\r)^{1/2}\\
%&\leq \l( \int_{\R^N} |u(x)|^{2(\rho -1)} dx\r)^{1/2}\| \nabla u\|_{L^2}\\
&\le \| u\|_{L^{2\rho-2}}^{\rho -1}\| \nabla u\|_{L^2}. 
\end{align*}
Here we set $2(\rho -1)=p+1$, which is equivalent that 
$
\rho := \frac{p+3}{2}.
$
Hence we deduce that
\begin{align*}
r^{N-1}|u(r)|^{\frac{p+3}{2}} \le \| u\|_{L^{p+1}}^{\frac{p+1}{2}} \|\nabla u\|_{L^2},
\end{align*}
which yields \eqref{eq:A.1}. 
\end{proof}
%%%%%%%%%%%%%%%%%%%%%%%
%%%%%%%%%%%%%%%%%%%%%%%

\begin{proof}[Proof of Proposition \ref{prop:A.1}]
If $u(x)$ is a nonincreasing function of $|x|$, we have
\begin{align*}
\| u\|_{L^{p+1}}^{p+1}\geq \int_{|x|<R} |u(x)|^{p+1}dx\geq | B_R |\, |u(R)|^{p+1}, 
%| \{ |x|<R\} |
\end{align*}
where $B_R:=\{ x\in\R^N :|x|<R\}$.
This yields that
\begin{align*}
\sup_{x\in\R^N}|x|^{\frac{N}{p+1}}|u(x)| \leq C\| u\|_{L^{p+1}}
\end{align*}
for some constant $C$. Hence, from the assumption and Lemma \ref{lem:A.2}, we deduce that $u_n(x)\to 0$ as $|x|\to\infty$ uniformly in $n\in\N$.

From the weak compactness, there exist $u\in \dot{H}^1(\R^N)\cap L^{p+1}(\R^N)$ and a subsequence $\{ u_{n_j}\}$ such that $u_{n_j}\wto u$ in $\dot{H}^1(\R^N)\cap L^{p+1}(\R^N)$ as $j\to\infty$. Fix $\eps >0$ and let $R>0$ to be chosen later. Giving $r$ as in the statement, we have
\begin{align*}
\| u_{n_j}-u \|_{L^r(\R^N)}^r&=\| u_{n_j}-u \|_{L^r(B_R)}^r +\| u_{n_j}-u \|_{L^r(|x|\geq R)}^r\\
%&= \| u_{n_j}-u \|_{L^r(B_R)}^r+ 
%\int_{|x|\geq R}| u_{n_j}-u|^{r-(p+1)+p+1} dx
%\\
&\leq \| u_{n_j}-u \|_{L^r(B_R)}^r 
+\| u_{n_j}-u \|_{L^\infty(|x|\geq R)}^{r-(p+1)}\| u_{n_j}-u\|_{L^{p+1}(\R^N)}^{p+1}.
\end{align*}  
We take $R$ large enough such that
\begin{align*}
\| u_{n_j}-u \|_{L^\infty(|x|\geq R)}^{r-(p+1)}\| u_{n_j}-u\|_{L^{p+1}(\R^N)}^{p+1} 
<\frac{\eps}{2}.
\end{align*}
We note that $R$ does not depend on $j$. Since $\{ u_{n_j}\}$ is bounded in 
$H^1(B_R)$ (due to $L^{p+1}(B_R)\subset L^2(B_R)$), from Rellich's compactness theorem we obtain that $\bigl. u_{n_j}\bigr|_{B_R}\to\bigl. u\bigr|_{B_R}$ in $L^r(B_R)$. Therefore for $j$ large enough we have
\begin{align*}
 \| u_{n_j}-u \|_{L^r(B_R)}^r <\frac{\eps}{2},
\end{align*} 
and so that $\| u_{n_j}-u \|_{L^r(\R^N)}^r<\eps$. This completes the proof.
\end{proof}

\section*{Acknowledgments}
M.H. would like to thank Masaya Maeda for helpful discussion about sharp decay estimates. N.F. was supported by JSPS KAKENHI Grant Number 20K14349.
M.H. was supported by JSPS KAKENHI Grant Number JP19J01504.

%%%%%%%%%%%%%%%%%%%%%%%%%%%%%%%%%%%%%%%
%%%%%%%%%%%%%%%%%%%%%%%%%%%%%%%%%%%%%%%

\end{document}